\documentclass[11pt,a4paper]{article}
\usepackage[utf8]{inputenc}
\usepackage[english]{babel}

\usepackage{fullpage}

\usepackage{graphicx}
\usepackage{amsmath}
\usepackage{amsfonts}
\usepackage{amssymb}
\usepackage{enumerate}
\usepackage{amsthm} 

\usepackage{caption}
\usepackage{subcaption}

\usepackage[dvipsnames]{xcolor}

\usepackage{tkz-graph} 
\usepackage{tkz-berge} 

\usepackage[ruled,vlined,linesnumbered]{algorithm2e} 
\SetKw{KwAnd}{and}

\newcounter{proofcase}
\newcounter{proofsubcase}[proofcase]
\newcounter{proofsubsubcase}[proofsubcase]
\newcounter{proofclaim}
\newcounter{proofblock}
\renewcommand{\theproofcase}{\arabic{proofcase}}
\renewcommand{\theproofsubcase}{\theproofcase.\arabic{proofsubcase}}
\renewcommand{\theproofsubsubcase}{\theproofsubcase.\arabic{proofsubsubcase}}
\renewcommand{\theproofclaim}{\Alph{proofclaim}}
\newcommand{\resetproofclaims}{\stepcounter{proofblock}\setcounter{proofclaim}{0}}
\newcommand{\resetproofcases}{\setcounter{proofcase}{0}\setcounter{proofsubcase}{0}\setcounter{proofsubsubcase}{0}}
\newcommand{\proofclaimspace}{\par\smallskip}
\newcommand{\proofcase}[1]{\stepcounter{proofcase}\setcounter{proofsubcase}{0}\par\medskip\indent\emph{Case~\theproofcase: #1.} }
\newcommand{\proofsubcase}[1]{\stepcounter{proofsubcase}\par\smallskip\indent\emph{Subcase~\theproofsubcase: #1.} }
\newcommand{\proofsubsubcase}[1]{\stepcounter{proofsubsubcase}\par\smallskip\indent\emph{Subcase~\theproofsubsubcase: #1.} }
\newcommand{\proofclaim}[1]{\refstepcounter{proofclaim}\proofclaimspace\indent\emph{Claim~\theproofclaim: #1.} }

\usepackage{hyperref}
\hypersetup{
    colorlinks=true,
    linkcolor=NavyBlue,
    citecolor=NavyBlue,
    urlcolor=NavyBlue
}

\newtheorem{theorem}{Theorem}[section]

\newtheorem{lemma}[theorem]{Lemma}
\newtheorem{proposition}[theorem]{Proposition}

\numberwithin{subcase}{case}

\numberwithin{subsubcase}{subcase}

\theoremstyle{definition}
\newtheorem{definition}[theorem]{Definition}
\newtheorem{remark}[theorem]{Remark}

\usetikzlibrary{arrows,backgrounds,calc,trees}

\pgfdeclarelayer{background}
\pgfsetlayers{background,main}

\DeclareMathOperator{\nil}{\emptyset}
\DeclareMathOperator{\cert}{cert}
\DeclareMathOperator{\DHG}{DH}
\DeclareMathOperator{\OVE}{OVE}
\DeclareMathOperator{\TS}{TS}
\newcommand{\flabel}{\mathsf{f}}
\newcommand{\tlabel}{\mathsf{t}}
\newcommand{\plabel}{\mathsf{p}}

\title{Characterization and linear-time recognition of balanced distance-hereditary graphs}
\author{
Luc\'ia Busolini\thanks{Departamento de Matem\'atica, Facultad de Ciencias Exactas y Naturales, Universidad de Buenos Aires, Buenos Aires, Argentina and Instituto de C\'alculo, CONICET-Universidad de Buenos Aires, Buenos Aires, Argentina. E-mail: \texttt{lbusolini@dm.uba.ar}}
\and
Guillermo Dur\'an\thanks{Departamento de Matem\'atica, Facultad de Ciencias Exactas y Naturales, Universidad de Buenos Aires, Buenos Aires, Argentina and Instituto de C\'alculo, CONICET-Universidad de Buenos Aires, Buenos Aires, Argentina. E-mail: \texttt{gduran@dm.uba.ar}}
\and
Mart\'in D.\ Safe\thanks{Departamento de Matem\'atica, Universidad Nacional del Sur (UNS), Bah\'ia Blanca, Argentina and INMABB, Universidad Nacional del Sur (UNS)-CONICET, Bah\'ia Blanca, Argentina. E-mail: \texttt{msafe@uns.edu.ar}}
}
\date{}

\begin{document}

\maketitle

\begin{abstract}
A graph is balanced if its clique-matrix contains no square submatrix of odd order with exactly two $1$'s in each row and in each column. Although it is known that a graph is balanced if and only if it contains no induced extended odd sun, a characterization of balanced graphs by minimal forbidden induced subgraphs is still unknown. In this work, we prove that, within the class of distance-hereditary graphs, balanced graphs are exactly the hereditary clique-Helly graphs. Equivalently, they are characterized by a single forbidden induced subgraph, namely $\overline{3K_2}$. From this result, we derive an explicit linear-time algorithm that decides balancedness within the class of distance-hereditary graphs and returns an induced $\overline{3K_2}$ when the input graph is not balanced.
\end{abstract}

\section{Introduction}\label{sec:intro}

Let $G$ be a graph. Let $Q_1,\ldots,Q_k$ be the maximal cliques and let $v_1,\ldots,v_n$ be the vertices of $G$. A \emph{clique-matrix} of $G$ is the matrix $A=(a_{ij})$ whose rows are indexed by $Q_1,\ldots,Q_k$, whose columns are indexed by $v_1,\ldots,v_n$, and such that $a_{ij}=1$ if $v_j\in Q_i$ and $a_{ij}=0$ otherwise. A graph $G$ is \emph{balanced}~\cite{definicionBalanceado} if its clique-matrix contains no square submatrix of odd order with exactly two $1$'s in each row and in each column. The name ‘balanced graphs’ was introduced by Berge and Chvátal~\cite{definicionBalanceado}. Berge and Las Vergnas~\cite{MR266787-BergeLasVergnas} proved that balanced graphs are perfect and clique-perfect, and Berge~\cite{Berge} proved that they are also hereditary clique-Helly (see Section~\ref{sec:prelim}).

Since balanced graphs are perfect, they have no odd holes and no odd antiholes, where an \emph{odd hole} in a graph $G$ is a chordless cycle of $G$ whose length is odd and at least $5$ and an \emph{odd antihole} in $G$ is an odd hole in the complement $\overline{G}$. Moreover, Lehel and Tuza~\cite{LehelTuza} showed that balanced graphs contain no induced odd suns. Finally, Bonomo et al.\ \cite{OnBalancedGraphs} characterized balanced graphs by means of extended odd suns, a family of graphs that generalizes odd holes and odd suns.

\begin{theorem}[\cite{OnBalancedGraphs}] A graph is balanced if and only if it contains no induced extended odd sun.\end{theorem}

However, this characterization is not by minimal forbidden induced subgraphs. Indeed, some extended odd suns contain other extended odd suns as proper induced subgraphs (see Figure~\ref{fig:extoddsunsinother}). In fact, the characterization of balanced graphs by minimal forbidden induced subgraphs remains unknown, but some partial results are known~\cite{MR4103815}.

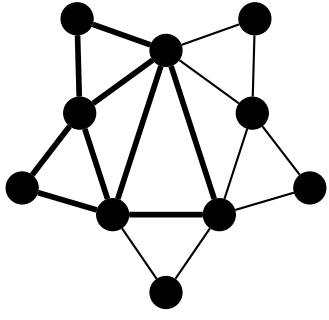
\begin{figure}
    \centering
    \begin{tikzpicture}[scale=0.4]
       \SetVertexNoLabel
       \GraphInit[vstyle=Classic]

        \begin{scope}[rotate=18]
        \begin{scope}[rotate=36]
        \grEmptyCycle[prefix=v,RA=5]{5}
        \end{scope}
        \grEmptyCycle[prefix=w,RA=3]{5}
        \EdgeInGraphLoop{w}{5}
        \end{scope}

        \Edges(w0,v0, w1)
        \Edges(w1,v1, w2)
        \Edges(w2,v2, w3)
        \Edges(w3,v3, w4)
        \Edges(w4,v4, w0)

        \Edges(w3, w1, w4)

        \tikzstyle{EdgeStyle}=[line width=0.75mm]
        \Edges(w1,v1,w2, v2, w3, w4, w1,w3,w2,w1)
    \end{tikzpicture}
    \caption{An extended odd sun that is not a minimal induced subgraph for the class of balanced graphs. Bold lines are the edges of a proper induced extended odd sun.}
    \label{fig:extoddsunsinother}
\end{figure}

Balanced graphs were characterized by minimal forbidden induced subgraphs when restricted to diamond-free graphs~\cite{Apollonio}, some subclasses of circular-arc graphs~\cite{BalancednessCircularArcGraphs}, paw-free graphs~\cite{CliquePerfectnessAndBalancedness}, $P_4$-tidy graphs~\cite{CliquePerfectnessAndBalancedness}, claw-free graphs~\cite{BusoliniDuranSafe2023ClawFreeBalancedGraphs}, complements of bipartite graphs, line graphs of multigraphs, and complements of line graphs of multigraphs~\cite{OnMinimalForbiddenSubgraphs}.

Balanced graphs can be recognized in polynomial time by applying Zambelli's recognition algorithm for balanced matrices~\cite{MR2154177}. More precisely, as observed in~\cite{OnMinimalForbiddenSubgraphs}, the balancedness of a graph with $n$ vertices and $m$ edges can be decided in $O(m^9+n)$ time. Within some graph classes, however, structural characterizations lead to linear-time algorithms for recognizing balancedness. This is the case for $P_4$-tidy graphs and paw-free graphs~\cite{CliquePerfectnessAndBalancedness}, complements of bipartite graphs, line graphs of multigraphs, and complements of line graphs of multigraphs~\cite{OnMinimalForbiddenSubgraphs}.

Since the class of $P_4$-tidy graphs contains all \emph{cographs} (i.e., all $P_4$-free graphs), the results of~\cite{CliquePerfectnessAndBalancedness} imply that, within the class of cographs, a graph is balanced if and only if it is hereditary clique-Helly and balancedness of cographs is characterized by the single forbidden induced subgraph $\overline{3K_2}$. Moreover, the same work implies a linear-time recognition algorithm for balancedness of cographs. The graph $\overline{3K_2}$ is shown in Figure~\ref{fig:co-3K2}.

\begin{figure}
    \centering
    \begin{tikzpicture}[scale=0.35]
        \centering
        \SetVertexNoLabel
        \GraphInit[vstyle=Classic]
        \grTriangularGrid{3}
        \Edge[style={bend left}](a0;2)(a2;0)
        \Edge[style={bend left}](a2;0)(a0;0)
        \Edge[style={bend left}](a0;0)(a0;2)
    \end{tikzpicture}
    \caption{The graph $\overline{3K_2}$.}
    \label{fig:co-3K2}
\end{figure}
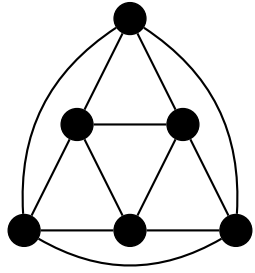

We study balancedness in the class of distance-hereditary graphs. A graph is \emph{distance-hereditary} if, in every connected induced subgraph, each pair of vertices is at the same distance as in the original graph~\cite{MR485544}. Distance-hereditary graphs are known to be perfect~\cite{MR485544} and clique-perfect~\cite{MR2203202}. The class of distance-hereditary graphs contains all cographs, but it is neither a subclass nor a superclass of the class of $P_4$-tidy graphs. For instance, $C_5$ is $P_4$-tidy but not distance-hereditary, whereas $P_6$ is distance-hereditary but not $P_4$-tidy. Thus, the results for $P_4$-tidy graphs do not cover the distance-hereditary case considered here.

In this work, we extend the balancedness results for cographs to the class of distance-hereditary graphs. More precisely, we prove that, within this larger class, balanced graphs are exactly the hereditary clique-Helly graphs and are characterized by the same single forbidden induced subgraph, namely $\overline{3K_2}$. From this result, we derive an explicit linear-time algorithm that decides balancedness within the class of distance-hereditary graphs and returns an induced $\overline{3K_2}$ whenever the input graph is not balanced.

This work is organized as follows. In Section~\ref{sec:prelim}, we introduce the graph and hypergraph terminology and preliminary results used throughout the paper. In Section~\ref{sec:dh-graphs}, we recall characterizations and recognition results for distance-hereditary graphs, including their characterization in terms of one-vertex-extension trees. In Section~\ref{sec:bicolorings}, we derive from Berge's bicoloring characterization of balanced hypergraphs a corresponding bicoloring characterization of balanced graphs and then adapt it to the case of graphs with false twins. In Section~\ref{sec:dh-balanced}, we prove the forbidden induced subgraph characterization of balancedness within the class of distance-hereditary graphs. In Section~\ref{sec:algorithmics}, we give a linear-time algorithm to recognize balancedness within the class of distance-hereditary graphs.

\section{Preliminaries}\label{sec:prelim}

\emph{Graphs.} All graphs in this work are finite, undirected, and without loops or multiple edges. If $G$ is a graph, we denote its vertex set and edge set by $V(G)$ and $E(G)$, respectively. If $S\subseteq V(G)$, we denote by $G-S$ the graph that arises from $G$ by removing all vertices in $S$. The \emph{subgraph of $G$ induced by $S$}, denoted by $G[S]$, is defined as $G-(V(G)-S)$. If $v\in V(G)$, we write $G-v$ instead of $G-\{v\}$. We denote by $\overline{G}$ the complement of $G$. If $G_1$ and $G_2$ are vertex-disjoint graphs, then $G_1\cup G_2$ denotes their disjoint union. If $r$ is a positive integer and $G$ is a graph, $rG$ denotes the disjoint union of $r$ copies of $G$.

Let $G$ be a graph. We denote by $N_G(v)$ the \emph{neighborhood} of a vertex $v$ in $G$ and by $N_G[v]$ its \emph{closed neighborhood} in $G$, which is $N_G(v)\cup\{v\}$. A vertex is \emph{pendant} if its degree is one. A vertex $v$ of $G$ is \emph{universal} if $N_G[v]=V(G)$. Two distinct vertices $u$ and $v$ of $G$ are called \emph{false twins} if $N_G(u)=N_G(v)$ and \emph{true twins} if $N_G[u]=N_G[v]$. In particular, true twins are adjacent to each other, while false twins are not.

A \emph{clique} in a graph is a set of pairwise adjacent vertices. A \emph{maximal clique} is an inclusion-wise maximal clique.

Let $G$ be a graph. If $A\subseteq V(G)$, a vertex $v\in V(G)$ is \emph{complete to} $A$ in $G$ if $A\subseteq N_G(v)$. If $A,B\subseteq V(G)$, then $B$ is \emph{complete to} $A$ in $G$ if $b$ is complete to $A$ in $G$ for each $b\in B$.

A \emph{path} $P$ in a graph $G$ is a sequence $P=v_1,v_2,\ldots,v_k$ of distinct vertices of $G$ such that $v_i$ and $v_{i+1}$ are adjacent in $G$ for each $i\in\{1,2,\ldots,k-1\}$. It is called a \emph{path on $k$ vertices} and its \emph{length} is $k-1$. If $x$ and $y$ are vertices of $G$ in the same connected component, the \emph{distance} between $x$ and $y$ in $G$ is the length of a shortest path of $G$ joining $x$ and $y$. A \emph{chord} of $P$ is an edge of $G$ joining vertices $v_i$ and $v_j$ for two distinct nonconsecutive $i,j\in\{1,2,\ldots,k\}$. The path $P$ is \emph{chordless} if it has no chords. We denote the chordless path, the chordless cycle, and the complete graph on $n$ vertices by $P_n$, $C_n$, and $K_n$, respectively. A \emph{hole} in a graph $G$ is a chordless cycle of $G$ of length at least $4$ and an \emph{antihole} in $G$ is a hole in $\overline{G}$. For $n\geq 3$, the \emph{wheel} $W_n$ is the graph that arises from $C_n$ by adding one vertex adjacent to every vertex of $C_n$.

If $G$ and $H$ are graphs, we say that $G$ is \emph{$H$-free} if $G$ has no induced subgraph isomorphic to $H$. A class $\mathcal{G}$ of graphs is \emph{hereditary} if, for every graph $G\in\mathcal{G}$, each induced subgraph of $G$ belongs to $\mathcal{G}$. If $\mathcal{G}$ is a hereditary class of graphs, then a graph $H$ is a \emph{minimal forbidden induced subgraph} for $\mathcal{G}$ if $H\notin\mathcal{G}$ and every proper induced subgraph of $H$ belongs to $\mathcal{G}$.

A graph $G$ is \emph{perfect}~\cite{Berge-perf2} if, in every induced subgraph of $G$, the chromatic number equals the size of a maximum clique. That perfect graphs are precisely the graphs having no odd holes and no odd antiholes was conjectured by Berge around 1960~\cite{Berge-perf2} and proved by Chudnovsky et al.\ \cite{SPGT}.

A \emph{clique-transversal} of a graph $G$ is a set of vertices of $G$ meeting every maximal clique of $G$. A \emph{clique-independent set} of a graph $G$ is a set of pairwise disjoint maximal cliques of $G$. A graph $G$ is \emph{clique-perfect}~\cite{MR1737764} if, in every induced subgraph of $G$, the minimum size of a clique-transversal equals the maximum size of a clique-independent set. For clique-perfect graphs, neither a characterization by minimal forbidden induced subgraphs nor a polynomial-time recognition algorithm is known, although several partial results are available (see~\cite{MR4103815}).

A graph $G$ is \emph{clique-Helly} if every family of pairwise intersecting maximal cliques of $G$ has nonempty total intersection. A graph $G$ is \emph{hereditary clique-Helly}~\cite{PrimeraDefinicionHCH} if every induced subgraph of $G$ is clique-Helly. Prisner characterized hereditary clique-Helly graphs by four six-vertex minimal forbidden induced subgraphs~\cite{PrimeraDefinicionHCH}. The following result of Berge is used in the proof of our main theorem.

\begin{theorem}[{\cite[Proposition~7]{Berge}}]\label{thm:balanced-HCH} Every balanced graph is hereditary clique-Helly.\end{theorem}

For graph terminology not defined here, the reader is referred to~\cite{MR1367739-West}.

\smallskip\noindent \emph{Matrices and hypergraphs.} A \emph{$\{0,1\}$-matrix} is a matrix whose entries belong to $\{0,1\}$. A $\{0,1\}$-matrix is \emph{balanced}~\cite{Berge} if it contains no square submatrix of odd order with exactly two $1$'s in each row and in each column.

A \emph{hypergraph}~\cite{MR0357172} is an ordered pair $H=(V,\mathcal{E})$ where $V=\{v_1,v_2,\ldots,v_n\}$ is a finite set, called the \emph{vertices}, and $\mathcal{E}=\{E_i\colon i\in I\}$ is a family of nonempty subsets of $V$ such that $\bigcup_{i\in I}E_i=V$. The elements of $\mathcal{E}$ are called the \emph{edges} of $H$. The \emph{subhypergraph of $H$ induced by} $S\subseteq V$ is the hypergraph $H_S=(S,\mathcal{E}_S)$ where $\mathcal{E}_S=\{E_i\cap S\colon i\in I \text{ and } E_i\cap S\neq\emptyset\}$.

The \emph{clique hypergraph} of a graph $G$ is the hypergraph $\mathcal{C}(G)$ whose vertex set is $V(G)$ and whose edge family is the set of maximal cliques of $G$. If $H$ is a hypergraph, an \emph{incidence matrix of $H$} is a $\{0,1\}$-matrix having a row for each edge and a column for each vertex such that an entry is $1$ if the vertex corresponding to its column belongs to the edge corresponding to its row.

A hypergraph $H$ is \emph{balanced}~\cite{Berge} if an incidence matrix of $H$ is balanced. A clique-matrix of a graph $G$ is an incidence matrix of $\mathcal{C}(G)$. Thus, a graph $G$ is balanced if and only if $\mathcal{C}(G)$ is balanced.

A \emph{$k$-coloring} of a hypergraph $H$ is an assignment of one of the colors $1,2,\ldots,k$ to each of its vertices. A \emph{proper $k$-coloring} of $H$ is a $k$-coloring such that no edge of $H$ with at least two vertices is monochromatic.

Berge proved the following characterization of balanced hypergraphs in terms of proper $2$-colorings.

\begin{theorem}[\cite{Berge}] \label{thm:balancedhypergraphs} A hypergraph $H=(V,\mathcal{E})$ is balanced if and only if, for every $S\subseteq V$, the induced subhypergraph $H_S$ admits a proper $2$-coloring.\end{theorem}

\section{Characterizations and recognition of distance-hereditary graphs}\label{sec:dh-graphs}

Distance-hereditary graphs were introduced by Howorka~\cite{MR485544}, who also proved that they are perfect. In this section, we recall some characterizations and recognition results for this class. We first state standard characterizations in terms of one-vertex extensions and forbidden induced subgraphs. We then recall the characterization of Chang et al.\ \cite{MR1651039} in terms of $\OVE$-trees, which provides the recursive representation used throughout the rest of the paper.

\begin{theorem}[\cite{MR859310,MR1055593}] \label{thm:DHequivDefs} For a graph $G$, the following assertions are equivalent:
\begin{enumerate}[(i)]
\item $G$ is distance-hereditary;
\item \label{item:DHrecursiveDef1} the vertices of $G$ admit an ordering $v_1,\ldots,v_n$ such that, for each $i\in\{2,\ldots,n\}$, the graph $G[\{v_1,\ldots,v_i\}]$ arises from $G[\{v_1,\ldots,v_{i-1}\}]$ by adding $v_i$ as a pendant vertex, as a true twin of one of its vertices, or as a false twin of one of its vertices;
\item $G$ contains no induced house, gem, or domino (shown in Figure~\ref{fig:forbiddenDH}) and has no hole of length at least $5$.
\end{enumerate}
\end{theorem}

\begin{figure}
    \centering
    \begin{subfigure}[b]{0.3\textwidth}
            \centering
                \begin{tikzpicture}[scale=0.35, rotate=45]
                    \centering
                    \SetVertexNoLabel
                    \GraphInit[vstyle=Classic]
                    \grCycle{4}
                    \EA[unit=4](a1){a4}
                    \Edges(a1, a4, a0)
        \end{tikzpicture}
        \caption*{house}
    \end{subfigure}
    \begin{subfigure}[b]{0.3\textwidth}
            \centering
                \begin{tikzpicture}[scale=0.8]
                    \centering
                    \SetVertexNoLabel
                    \GraphInit[vstyle=Classic]
                    \Vertex[x=0, y=0]{a0}
                    \Vertex[x=-2, y=1.7]{a1}
                    \Vertex[x=-1, y=3]{a2}
                    \Vertex[x=1, y=3]{a3}
                    \Vertex[x=2, y=1.7]{a4}
                    \Edges(a0, a1, a2, a3, a4, a0)
                    \Edges(a2, a0, a3)
        \end{tikzpicture}
        \caption*{gem}
    \end{subfigure}
    \begin{subfigure}[b]{0.3\textwidth}
            \centering
                \begin{tikzpicture}[scale=0.7, rotate=90]
                    \centering
                    \SetVertexNoLabel
                    \GraphInit[vstyle=Classic]
                    \grLadder[RA=2,RS=2]{3}
        \end{tikzpicture}
        \caption*{domino}
    \end{subfigure}
    \caption{The house, the gem, and the domino.}
    \label{fig:forbiddenDH}
\end{figure}
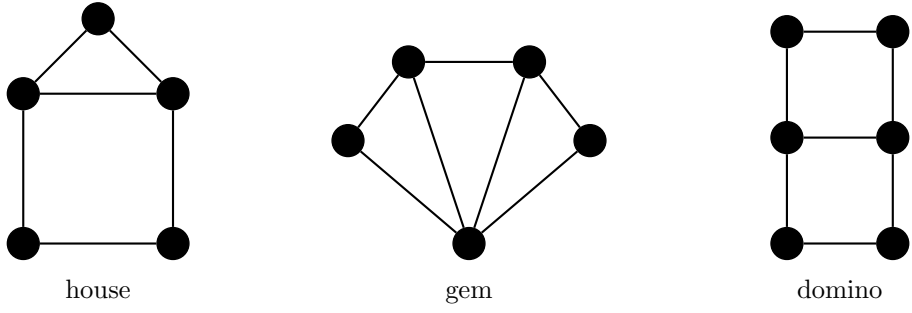

Chang et al.~\cite{MR1651039} gave another characterization of distance-hereditary graphs in terms of one-vertex-extension trees. The characterization is based on the following recursively defined trees.

\begin{definition}[\cite{MR1651039}]\label{def:OVE-tree} A \emph{one-vertex-extension tree}, or \emph{$\OVE$-tree}, is a finite rooted ordered tree whose edges are labeled in $\{\flabel,\tlabel,\plabel\}$ and for which exactly one of the following statements holds:
\begin{enumerate}[(i)]
\item\label{it:OVE1} $T$ has exactly one vertex.
\item\label{it:OVE2} There exist vertex-disjoint $\OVE$-trees $T_1$ and $T_2$ and a label $\alpha\in\{\flabel,\tlabel,\plabel\}$ such that $T$ is obtained from the disjoint union of $T_1$ and $T_2$ by joining the root of $T_1$ to the root of $T_2$ with an edge labeled $\alpha$, taking the root of $T_1$ as the root of $T$, making the root of $T_2$ the first child of the root of $T_1$ and keeping all other child orders as in $T_1$ and $T_2$.
\end{enumerate}
When \eqref{it:OVE2} holds, we write $T=T_1\oplus_\alpha T_2$.\end{definition}

We associate with each $\OVE$-tree $T$ a graph $\DHG(T)$ and a vertex set $\TS(T)$.

\begin{definition}[\cite{MR1651039}]\label{def:OVE-graph} Let $T$ be an $\OVE$-tree. We define $\DHG(T)$ and $\TS(T)$ recursively as follows:
\begin{enumerate}[(i)]
\item If $T$ has exactly one vertex $v$, then $\DHG(T)$ is the graph with vertex set $\{v\}$ and no edges and $\TS(T)=\{v\}$.
\item If $T=T_1\oplus_{\flabel} T_2$, then $\DHG(T)=\DHG(T_1)\cup\DHG(T_2)$ and $\TS(T)=\TS(T_1)\cup\TS(T_2)$.
\item If $T=T_1\oplus_{\tlabel} T_2$, then $\DHG(T)$ is the graph that arises from $\DHG(T_1)\cup\DHG(T_2)$ by adding all edges $uv$ with $u\in\TS(T_1)$ and $v\in\TS(T_2)$ and $\TS(T)=\TS(T_1)\cup\TS(T_2)$.
\item If $T=T_1\oplus_{\plabel} T_2$, then $\DHG(T)$ is the graph that arises from $\DHG(T_1)\cup\DHG(T_2)$ by adding all edges $uv$ with $u\in\TS(T_1)$ and $v\in\TS(T_2)$ and $\TS(T)=\TS(T_1)$.
\end{enumerate}
The graph $\DHG(T)$ is called the \emph{graph associated with $T$} and the set $\TS(T)$ is called the \emph{twin-set associated with $T$}. If $T=T_1\oplus_\alpha T_2$, we say that $\DHG(T)$ arises from $\DHG(T_1)$ and $\DHG(T_2)$ by a \emph{false twin operation} if $\alpha=\flabel$, by a \emph{true twin operation} if $\alpha=\tlabel$, and by a \emph{pendant vertex operation} if $\alpha=\plabel$.\end{definition}

\begin{remark}\label{rmk:TS-path} Let $T$ be an $\OVE$-tree and let $G=\DHG(T)$. Notice that $V(G)=V(T)$. Moreover, if $T=T_1\oplus_\alpha T_2$ with $\alpha\in\{\flabel,\tlabel,\plabel\}$ and $G_i=\DHG(T_i)$ for each $i\in\{1,2\}$, then \[ G_i=G[V(T_i)]\quad\text{for each } i\in\{1,2\}. \] Also, the root of $T$ belongs to $\TS(T)$.\end{remark}

Chang et al.\ characterized distance-hereditary graphs as the graphs associated with $\OVE$-trees.

\begin{theorem}[\cite{MR1651039}]\label{thm:NewDefinitionDH} A graph $G$ is distance-hereditary if and only if $G$ is isomorphic to $\DHG(T)$ for some $\OVE$-tree $T$.\end{theorem}

If $G=\DHG(T)$ for an $\OVE$-tree $T$, we say that $T$ \emph{is an $\OVE$-tree of $G$}. Hammer and Maffray~\cite{MR1055593} proposed a linear-time algorithm to produce an ordering satisfying Theorem~\ref{thm:DHequivDefs}\eqref{item:DHrecursiveDef1}. Damiand et al.~\cite{MR1846920} later reported a counterexample to the algorithm proposed in~\cite{MR1055593} and gave a corrected linear-time algorithm producing such an ordering, in the equivalent form of a pruning sequence. Chang et al.~\cite{MR1651039} gave a linear-time algorithm that, from such an ordering, builds an $\OVE$-tree of $G$. Thus, the result below follows.

\begin{theorem}[\cite{MR1651039,MR1846920,MR1055593}]\label{thm:AlgoRecognitionDH} There is a linear-time algorithm that, given a graph $G$, decides whether $G$ is distance-hereditary and, if so, returns an $\OVE$-tree of $G$.\end{theorem}

\section[Balanced graphs and proper 2-colorings]{Balanced graphs and proper $2$-colorings}\label{sec:bicolorings}

In this section, we first derive from Berge's bicoloring characterization of balanced hypergraphs a bicoloring characterization of balanced graphs (Lemma~\ref{lemma:proper2coloring}). We then adapt this characterization to the case of graphs with false twins, showing that it suffices to consider only subsets containing one of the false twins and not the other (Lemma~\ref{lemma:WconVynoV'}). These characterizations are the main tools used to prove balancedness throughout the paper.

If $G$ is a graph and $W\subseteq V(G)$, we define a \emph{proper $2$-coloring of $W$ with respect to $G$} to be a proper $2$-coloring of the induced subhypergraph $\mathcal{C}(G)_W$. Equivalently, it is a $2$-coloring $\phi$ of $W$ such that, for every maximal clique $Q$ of $G$ with $\vert Q\cap W\vert\geq 2$, $Q\cap W$ is not monochromatic under $\phi$.

\begin{lemma} \label{lemma:proper2coloring} A graph $G$ is balanced if and only if, for every $W\subseteq V(G)$, there is a proper $2$-coloring of $W$ with respect to $G$.\end{lemma}
\begin{proof} By definition, $G$ is balanced if and only if $\mathcal{C}(G)$ is balanced. By Theorem~\ref{thm:balancedhypergraphs}, this is equivalent to saying that, for every $W\subseteq V(G)$, the induced subhypergraph $\mathcal{C}(G)_W$ admits a proper $2$-coloring. It remains only to observe that a proper $2$-coloring of $\mathcal{C}(G)_W$ is exactly a proper $2$-coloring of $W$ with respect to $G$.\end{proof}

We now adapt the preceding characterization to graphs with false twins.

\begin{lemma}\label{lemma:WconVynoV'} Let $G$ be a graph and let $v$ and $v'$ be false twins of $G$. The graph $G$ is balanced if and only if every set $W\subseteq V(G)$ such that $v\in W$ and $v'\notin W$ admits a proper $2$-coloring with respect to $G$.\end{lemma}
\begin{proof} The `only if' implication follows immediately from Lemma~\ref{lemma:proper2coloring}. To prove the `if' implication, assume that each subset $W\subseteq V(G)$ such that $v\in W$ and $v'\notin W$ admits a proper $2$-coloring with respect to $G$. We prove that, as a consequence, such a coloring also exists for every subset $W\subseteq V(G)$ such that $v\notin W$ or $v'\in W$. Let $W$ be such a subset. We consider the three possible cases.
\resetproofcases
\proofcase{$v\notin W$ and $v'\in W$} Let $W_0=(W-\{v'\})\cup\{v\}$. Since $v\in W_0$ and $v'\notin W_0$, our assumption yields a proper $2$-coloring $\phi_0$ of $W_0$ with respect to $G$. Let $\phi$ be the $2$-coloring of $W$ such that $\phi\vert_{W-\{v'\}}=\phi_0\vert_{W-\{v'\}}$ and $\phi(v')=\phi_0(v)$. Let $Q$ be a maximal clique of $G$ such that $\vert Q\cap W\vert\geq 2$. Suppose first that $Q$ contains neither $v$ nor $v'$. Thus, $Q\cap W=Q\cap W_0$ and this set has at least two vertices. Since $\phi_0$ is proper and $\phi$ agrees with $\phi_0$ on $Q\cap W$, $Q\cap W$ is not monochromatic under $\phi$. Suppose now that $v'\in Q$. Let $Q_0=(Q-\{v'\})\cup\{v\}$. Since $v$ and $v'$ are false twins in $G$, $Q_0$ is a maximal clique of $G$. Since $Q_0\cap W_0$ is obtained from $Q\cap W$ by replacing $v'$ with $v$, we have $\vert Q_0\cap W_0\vert=\vert Q\cap W\vert\geq 2$. By the properness of $\phi_0$, $Q_0\cap W_0$ is not monochromatic under $\phi_0$. By the definition of $\phi$, this implies that $Q\cap W$ is not monochromatic under $\phi$. Suppose finally that $v\in Q$. Let $Q_0=(Q-\{v\})\cup\{v'\}$. Since $v$ and $v'$ are false twins in $G$, $Q_0$ is a maximal clique of $G$ and $Q_0\cap W_0=Q\cap W$. Hence, by the properness of $\phi_0$ and since $\phi$ and $\phi_0$ agree on $Q\cap W$, $Q\cap W$ is not monochromatic under $\phi$. Therefore, $\phi$ is a proper $2$-coloring of $W$ with respect to $G$.

\proofcase{$v\in W$ and $v'\in W$} Let $W_0=W-\{v'\}$. Since $v\in W_0$ and $v'\notin W_0$, our assumption yields a proper $2$-coloring $\phi_0$ of $W_0$ with respect to $G$. Let $\phi$ be the $2$-coloring of $W$ such that $\phi\vert_{W_0}=\phi_0$ and $\phi(v')=\phi_0(v)$. Let $Q$ be a maximal clique of $G$ such that $\vert Q\cap W\vert\geq 2$. Suppose first that $v'\notin Q$. Thus, $Q\cap W=Q\cap W_0$ and this set has at least two vertices. Since $\phi_0$ is proper and $\phi$ agrees with $\phi_0$ on $Q\cap W$, $Q\cap W$ is not monochromatic under $\phi$. Suppose otherwise that $v'\in Q$ and let $Q_0=(Q-\{v'\})\cup\{v\}$. Since $v$ and $v'$ are false twins in $G$, $Q_0$ is a maximal clique of $G$. Since $Q_0\cap W_0$ is obtained from $Q\cap W$ by replacing $v'$ with $v$, we have $\vert Q_0\cap W_0\vert=\vert Q\cap W\vert\geq 2$. By the properness of $\phi_0$, $Q_0\cap W_0$ is not monochromatic under $\phi_0$. Since $\phi$ agrees with $\phi_0$ on $(Q\cap W)-\{v'\}$ and $\phi(v')=\phi_0(v)$, $Q\cap W$ is not monochromatic under $\phi$. Therefore, $\phi$ is a proper $2$-coloring of $W$ with respect to $G$.

\proofcase{$v\notin W$ and $v'\notin W$} Let $W_0=W\cup\{v\}$. Since $v\in W_0$ and $v'\notin W_0$, our assumption yields a proper $2$-coloring $\phi_0$ of $W_0$ with respect to $G$. Let $\phi=\phi_0\vert_W$. Let $Q$ be a maximal clique of $G$ such that $\vert Q\cap W\vert\geq 2$. Suppose first that $v\notin Q$. Since $Q\cap W=Q\cap W_0$ and $\phi=\phi_0\vert_W$, the properness of $\phi_0$ implies that $Q\cap W$ is not monochromatic under $\phi$. Suppose otherwise that $v\in Q$ and let $Q_0=(Q-\{v\})\cup\{v'\}$. Since $v$ and $v'$ are false twins in $G$, $Q_0$ is a maximal clique of $G$. Since $Q_0\cap W_0=Q\cap W$ and $\phi=\phi_0\vert_W$, the properness of $\phi_0$ implies that $Q\cap W$ is not monochromatic under $\phi$. Therefore, $\phi$ is a proper $2$-coloring of $W$ with respect to $G$.

\medskip As we have proved that every subset $W\subseteq V(G)$ admits a proper $2$-coloring with respect to $G$, Lemma~\ref{lemma:proper2coloring} implies that $G$ is balanced. This completes the proof of the lemma.\end{proof}

\section{Structural characterization of balanced distance-hereditary graphs}\label{sec:dh-balanced}

In this section, we prove the main result of this work (Theorem~\ref{thm:DHbalancedGraphs}). The result states that, for distance-hereditary graphs, being balanced, being hereditary clique-Helly, and containing no induced $\overline{3K_2}$ are equivalent. Thus, within this class, balanced graphs are characterized by a single forbidden induced subgraph.

The `only if' implication of Theorem~\ref{thm:DHbalancedGraphs} is immediate. For the `if' implication, by Theorem~\ref{thm:NewDefinitionDH}, it suffices to consider graphs of the form $G=\DHG(T)$, where $T$ is an $\OVE$-tree. This allows us to prove this implication by induction on the definition of $\OVE$-trees. In the inductive step, when $T=T_1\oplus_\alpha T_2$ with $\alpha\in\{\tlabel,\plabel\}$, we express $G$ through Cunningham's composition for undirected graphs. To assess balancedness in these cases, we first prove Proposition~\ref{prop:ast-bal}, which characterizes when a graph obtained by this composition is balanced. We then prove Proposition~\ref{prop:TS-bal}, which provides assertions about the balancedness of the augmentations required when applying Proposition~\ref{prop:ast-bal}.

We recall Cunningham's composition for undirected graphs~\cite{MR655562}. Let $H_1$ and $H_2$ be graphs such that $V(H_1)\cap V(H_2)=\{v\}$. For each $i\in\{1,2\}$, let $G_i=H_i-v$ and $X_i=N_{H_i}(v)$. The \emph{composition of $H_1$ and $H_2$}, denoted $H_1\ast H_2$, is the graph that arises from $G_1\cup G_2$ by adding all edges with one endpoint in $X_1$ and the other in $X_2$.

Let $T=T_1\oplus_\alpha T_2$ with $\alpha\in\{\tlabel,\plabel\}$. Let $G=\DHG(T)$. For each $i\in\{1,2\}$, let $G_i=\DHG(T_i)$ and $X_i=\TS(T_i)$. If, for each $i\in\{1,2\}$, $H_i$ arises from $G_i$ by adding the same vertex $v$ with $N_{H_i}(v)=X_i$, then $G=H_1\ast H_2$.

Below, in Proposition~\ref{prop:ast-bal}, we characterize when $H_1\ast H_2$ is balanced. This result is relevant in two ways. First, it is key for proving balancedness in the cases of the proof of Theorem~\ref{thm:DHbalancedGraphs} where $T=T_1\oplus_\alpha T_2$ and $\alpha\in\{\tlabel,\plabel\}$. Second, Proposition~\ref{prop:ast-bal} is stated for arbitrary graphs $H_1$ and $H_2$ sharing a single common vertex $v$ and therefore gives a balancedness criterion for this composition beyond distance-hereditary graphs.

The proof of Proposition~\ref{prop:ast-bal} relies on the following lemma about the maximal cliques of $H_1\ast H_2$.

\begin{lemma}\label{lem:cliques-ast} Let $H_1$ and $H_2$ be graphs such that $V(H_1)\cap V(H_2)=\{v\}$. For each $i\in\{1,2\}$, let $G_i=H_i-v$ and $X_i=N_{H_i}(v)$. Suppose that $X_1$ and $X_2$ are both nonempty. If $Q$ is a maximal clique of $H_1\ast H_2$, then one of the following assertions holds:
\begin{enumerate}[(i)]
\item\label{it:Q1} $Q$ is a maximal clique of $G_i$ not contained in $X_i$, for some $i\in\{1,2\}$.
\item\label{it:Q2} $Q=R_1\cup R_2$, where $R_i$ is a maximal clique of $G_i[X_i]$ for each $i\in\{1,2\}$.
\end{enumerate}
\end{lemma}
\begin{proof} Let $G=H_1\ast H_2$ and let $Q$ be a maximal clique of $G$.

Suppose first that $Q\subseteq V(G_i)$ for some $i\in\{1,2\}$. Since $G_i$ is an induced subgraph of $G$, $Q$ is a maximal clique of $G_i$. Moreover, $Q$ is not contained in $X_i$, since otherwise a vertex of $X_{3-i}$ would be adjacent in $G$ to every vertex of $Q$, contradicting the maximality of $Q$ in $G$. Thus, in this case, assertion~\eqref{it:Q1} holds.

Suppose now that $Q$ is contained in neither $V(G_1)$ nor $V(G_2)$. For each $i\in\{1,2\}$, let $R_i=Q\cap V(G_i)$. Since $V(G_1)$ and $V(G_2)$ partition $V(G)$, we have $Q=R_1\cup R_2$, where $R_1$ and $R_2$ are both nonempty. Since $Q$ is a clique of $G$, $R_1$ is complete to $R_2$ in $G$. By the definition of $H_1\ast H_2$, this implies that $R_1\subseteq X_1$ and $R_2\subseteq X_2$. Since $G_i$ is an induced subgraph of $G$, $R_i$ is a clique of $G_i[X_i]$ for each $i\in\{1,2\}$. Finally, for each $i\in\{1,2\}$, $R_i$ is a maximal clique of $G_i[X_i]$, since any vertex of $X_i-R_i$ complete to $R_i$ in $G_i$ would also be complete to $R_{3-i}$ in $G$ (as $X_i$ is complete to $X_{3-i}$ in $G$ and $R_{3-i}\subseteq X_{3-i}$) and hence complete to $Q$ in $G$, contradicting the maximality of $Q$ in $G$. Therefore, in this case, assertion~\eqref{it:Q2} holds. This completes the proof of the lemma.\end{proof}

We are now ready to prove the following characterization of when $H_1\ast H_2$ is balanced.

\begin{proposition}\label{prop:ast-bal} Let $H_1$ and $H_2$ be graphs such that $V(H_1)\cap V(H_2)=\{v\}$. For each $i\in\{1,2\}$, let $X_i=N_{H_i}(v)$ and let $H_i^+$ be the graph that arises from $H_i$ by adding a false twin of $v$. If $X_1$ and $X_2$ are nonempty, then the following assertions hold:
\begin{enumerate}[(i)]
\item \label{item:TSareCliques} If $X_1$ is a clique of $H_1$ and $X_2$ is a clique of $H_2$, then $H_1\ast H_2$ is balanced if and only if $H_1$ and $H_2$ are balanced.
\item \label{item:TS1isClique} If $X_1$ is a clique of $H_1$ and $X_2$ is not a clique of $H_2$, then $H_1\ast H_2$ is balanced if and only if $H_1^+$ and $H_2$ are balanced.
\item \label{item:TSareNotCliques} If neither $X_1$ is a clique of $H_1$ nor $X_2$ is a clique of $H_2$, then $H_1\ast H_2$ is balanced if and only if $H_1^+$ and $H_2^+$ are balanced.
\end{enumerate}
\end{proposition}
\begin{proof} Throughout this proof, let $G=H_1\ast H_2$, let $G_i=H_i-v$ for each $i\in\{1,2\}$, and let $v'$ denote, for each $i\in\{1,2\}$, the false twin of $v$ in $H_i^+$.

Suppose that $G$ is balanced. Since $X_1$ and $X_2$ are nonempty, let $x_i$ be a vertex of $X_{3-i}$ for each $i\in\{1,2\}$. For each $i\in\{1,2\}$, $H_i$ is isomorphic to $G[V(G_i)\cup\{x_i\}]$, which is balanced as an induced subgraph of the balanced graph $G$. We now show that, for each $i\in\{1,2\}$, if $X_{3-i}$ is not a clique of $H_{3-i}$, then $H_i^+$ is balanced. Let $i\in\{1,2\}$ and suppose that $X_{3-i}$ is not a clique of $H_{3-i}$. Let $y$ and $y'$ be two vertices of $X_{3-i}$ that are nonadjacent in $H_{3-i}$. The graph $H_i^+$ is isomorphic to $G[V(G_i)\cup\{y,y'\}]$, which is balanced as an induced subgraph of the balanced graph $G$. Hence, $H_i^+$ is balanced. Thus, the `only if' implications in assertions~\eqref{item:TSareCliques},~\eqref{item:TS1isClique}, and~\eqref{item:TSareNotCliques} hold. It only remains to prove the `if' implication in each assertion.

For each of the three `if' implications below, let $W$ be an arbitrary subset of $V(G)$. By Lemma~\ref{lemma:proper2coloring}, to prove that $G$ is balanced under the hypotheses of the corresponding assertion, it suffices to give a proper $2$-coloring of $W$ with respect to $G$.

\par\medskip\indent\emph{Proof of the `if' implication in assertion~\eqref{item:TSareCliques}.} Suppose that $X_1$ is a clique of $H_1$ and $X_2$ is a clique of $H_2$. Suppose, in addition, that $H_1$ and $H_2$ are balanced. For each $i\in\{1,2\}$, let $W_i=W\cap V(G_i)$ if $W\cap X_{3-i}=\emptyset$ and let $W_i=(W\cap V(G_i))\cup\{v\}$ otherwise. Since $H_i$ is balanced for each $i\in\{1,2\}$, Lemma~\ref{lemma:proper2coloring} yields, for each $i\in\{1,2\}$, a proper $2$-coloring $\phi_i$ of $W_i$ with respect to $H_i$. If $v\in W_1\cap W_2$, interchange the colors of $\phi_2$ if necessary so that $\phi_1(v)\neq\phi_2(v)$. Let $\phi$ be the $2$-coloring of $W$ such that $\phi\vert_{W\cap V(G_i)}=\phi_i\vert_{W\cap V(G_i)}$ for each $i\in\{1,2\}$. This is well defined because $V(G_1)$ and $V(G_2)$ partition $V(G)$.

Let $Q$ be a maximal clique of $G$ such that $\vert Q\cap W\vert\geq 2$. Since $X_1$ is a clique of $G_1$ and $X_2$ is a clique of $G_2$, Lemma~\ref{lem:cliques-ast} implies that one of the following two cases holds for $Q$.

\resetproofcases
\proofcase{$Q$ is a maximal clique of $G_i$ not contained in $X_i$, for some $i\in\{1,2\}$} Hence, $Q\cup\{v\}$ is not a clique of $H_i$. Therefore, $Q$ is a maximal clique of $H_i$ and $Q\cap W=Q\cap W_i$. The properness of $\phi_i$ implies that $Q\cap W_i$ is not monochromatic under $\phi_i$. Since $\phi$ agrees with $\phi_i$ on $Q\cap W_i$, $Q\cap W$ is not monochromatic under $\phi$.

\proofcase{$Q=X_1\cup X_2$} Hence, $X_i\cup\{v\}$ is a maximal clique of $H_i$ for each $i\in\{1,2\}$, because $N_{H_i}(v)=X_i$. We consider the following two subcases.

\proofsubcase{$W\cap X_1=\emptyset$ or $W\cap X_2=\emptyset$} Let $i\in\{1,2\}$ be such that $W\cap X_{3-i}=\emptyset$. Since $Q=X_1\cup X_2$ and $\vert Q\cap W\vert\geq 2$, we have $Q\cap W=X_i\cap W$ and $\vert X_i\cap W\vert\geq 2$. By the definition, $W_i=W\cap V(G_i)$. Thus, $(X_i\cup\{v\})\cap W_i=X_i\cap W$. Since $X_i\cup\{v\}$ is a maximal clique of $H_i$, the properness of $\phi_i$ implies that $X_i\cap W$ is not monochromatic under $\phi_i$. Since $\phi$ agrees with $\phi_i$ on $X_i\cap W$, the equality $Q\cap W=X_i\cap W$ implies that $Q\cap W$ is not monochromatic under $\phi$.

\proofsubcase{$W\cap X_1\neq\emptyset$ and $W\cap X_2\neq\emptyset$} By definition, $v\in W_i$ for each $i\in\{1,2\}$. Hence, $(X_i\cup\{v\})\cap W_i$ has at least two vertices for each $i\in\{1,2\}$. Since $X_i\cup\{v\}$ is a maximal clique of $H_i$, the properness of $\phi_i$ ensures the existence of a vertex $x_i\in X_i\cap W$ such that $\phi_i(x_i)\neq\phi_i(v)$ for each $i\in\{1,2\}$. Since there are only two colors and $\phi_1(v)\neq\phi_2(v)$, it follows that $\phi(x_1)\neq\phi(x_2)$. Thus, $Q\cap W$ is not monochromatic under $\phi$.

\proofclaimspace The two cases for $Q$ above exhaust all possibilities and in each case $Q\cap W$ is not monochromatic under $\phi$. Since $Q$ is an arbitrary maximal clique of $G$ such that $\vert Q\cap W\vert\geq 2$, $\phi$ is a proper $2$-coloring of $W$ with respect to $G$. Since $W$ is an arbitrary subset of $V(G)$, Lemma~\ref{lemma:proper2coloring} implies that $G$ is balanced. This completes the proof of the `if' implication of assertion~\eqref{item:TSareCliques}.

\par\medskip\indent\emph{Proof of the `if' implication in assertion~\eqref{item:TS1isClique}.} Suppose that $X_1$ is a clique of $H_1$ and $X_2$ is not a clique of $H_2$. Suppose, in addition, that $H_1^+$ and $H_2$ are balanced. Let $W_1=(W\cap V(G_1))\cup\{v\}$. Let $W_2=W\cap V(G_2)$ if $W\cap X_1=\emptyset$ and let $W_2=(W\cap V(G_2))\cup\{v\}$ otherwise. Since $H_1^+$ and $H_2$ are balanced, Lemma~\ref{lemma:proper2coloring} yields a proper $2$-coloring $\phi_1$ of $W_1$ with respect to $H_1^+$ and a proper $2$-coloring $\phi_2$ of $W_2$ with respect to $H_2$. If $v\in W_2$, interchange the colors of $\phi_2$ if necessary so that $\phi_1(v)\neq\phi_2(v)$. Let $\phi$ be the $2$-coloring of $W$ such that $\phi\vert_{W\cap V(G_i)}=\phi_i\vert_{W\cap V(G_i)}$ for each $i\in\{1,2\}$.

Let $Q$ be a maximal clique of $G$ such that $\vert Q\cap W\vert\geq 2$. Since $X_1$ is a clique of $G_1$, Lemma~\ref{lem:cliques-ast} implies that one of the following two cases holds for $Q$.

\resetproofcases
\proofcase{$Q$ is a maximal clique of $G_i$ not contained in $X_i$, for some $i\in\{1,2\}$} If $i=1$, then neither $Q\cup\{v\}$ nor $Q\cup\{v'\}$ is a clique of $H_1^+$. Hence, $Q$ is a maximal clique of $H_1^+$. If $i=2$, then $Q\cup\{v\}$ is not a clique of $H_2$. Hence, $Q$ is a maximal clique of $H_2$. In either case, $Q\cap W=Q\cap W_i$. The properness of $\phi_i$ implies that $Q\cap W_i$ is not monochromatic under $\phi_i$. Since $\phi$ agrees with $\phi_i$ on $Q\cap W_i$, $Q\cap W$ is not monochromatic under $\phi$.

\proofcase{$Q=X_1\cup R_2$, where $R_2$ is a maximal clique of $G_2[X_2]$} Since $X_1$ is a clique of $H_1$, $X_1\cup\{v\}$ and $X_1\cup\{v'\}$ are maximal cliques of $H_1^+$, because each of $v$ and $v'$ is adjacent exactly to $X_1$ in $H_1^+$ and $v$ and $v'$ are nonadjacent in $H_1^+$. Also, $R_2\cup\{v\}$ is a maximal clique of $H_2$, since any vertex of $H_2$ complete to $R_2\cup\{v\}$ would belong to $X_2-R_2$, contradicting the maximality of $R_2$ in $G_2[X_2]$. We now consider subcases according to the size of $X_1\cap W$.

\proofsubcase{$\vert X_1\cap W\vert\geq 2$} Since $(X_1\cup\{v'\})\cap W_1=X_1\cap W$ and $X_1\cup\{v'\}$ is a maximal clique of $H_1^+$, the properness of $\phi_1$ implies that $X_1\cap W$ is not monochromatic under $\phi_1$. Since $X_1\cap W\subseteq Q\cap W$ and $\phi$ agrees with $\phi_1$ on $X_1\cap W$, $Q\cap W$ is not monochromatic under $\phi$.

\proofsubcase{$\vert X_1\cap W\vert=1$} Let $x_1$ be the unique vertex of $X_1\cap W$. Since $(X_1\cup\{v\})\cap W_1=\{x_1,v\}$, the properness of $\phi_1$ implies $\phi_1(x_1)\neq\phi_1(v)$. Since $Q=X_1\cup R_2$, the assumptions $\vert X_1\cap W\vert=1$ and $\vert Q\cap W\vert\geq 2$ imply that $R_2\cap W$ is nonempty. Moreover, since $W\cap X_1$ is nonempty, we have $v\in W_2$. Thus, the maximal clique $R_2\cup\{v\}$ of $H_2$ has at least two vertices in $W_2$. Hence, the properness of $\phi_2$ ensures the existence of a vertex $x_2\in R_2\cap W$ such that $\phi_2(x_2)\neq\phi_2(v)$. Since there are only two colors and $\phi_1(v)\neq\phi_2(v)$, it follows that $\phi(x_1)\neq\phi(x_2)$. Thus, $Q\cap W$ is not monochromatic under $\phi$.

\proofsubcase{$X_1\cap W=\emptyset$} In this subcase, $W_2=W\cap V(G_2)$ and $Q\cap W=R_2\cap W$. Since $\vert Q\cap W\vert\geq 2$, the maximal clique $R_2\cup\{v\}$ of $H_2$ has at least two vertices in $W_2$. The properness of $\phi_2$ implies that $Q\cap W$ is not monochromatic under $\phi$.

\proofclaimspace The two cases for $Q$ above exhaust all possibilities and in each case $Q\cap W$ is not monochromatic under $\phi$. Since $Q$ is an arbitrary maximal clique of $G$ such that $\vert Q\cap W\vert\geq 2$, $\phi$ is a proper $2$-coloring of $W$ with respect to $G$. Since $W$ is an arbitrary subset of $V(G)$, Lemma~\ref{lemma:proper2coloring} implies that $G$ is balanced. This completes the proof of the `if' implication of assertion~\eqref{item:TS1isClique}.

\par\medskip\indent\emph{Proof of the `if' implication in assertion~\eqref{item:TSareNotCliques}.} Suppose that neither $X_1$ is a clique of $H_1$ nor $X_2$ is a clique of $H_2$. Suppose, in addition, that $H_1^+$ and $H_2^+$ are balanced. For each $i\in\{1,2\}$, let $W_i=(W\cap V(G_i))\cup\{v\}$. Since $H_i^+$ is balanced for each $i\in\{1,2\}$, Lemma~\ref{lemma:proper2coloring} yields, for each $i\in\{1,2\}$, a proper $2$-coloring $\phi_i$ of $W_i$ with respect to $H_i^+$. Interchanging the colors of $\phi_2$ if necessary, we assume that $\phi_1(v)\neq\phi_2(v)$. Let $\phi$ be the $2$-coloring of $W$ such that $\phi\vert_{W\cap V(G_i)}=\phi_i\vert_{W\cap V(G_i)}$ for each $i\in\{1,2\}$.

Let $Q$ be a maximal clique of $G$ such that $\vert Q\cap W\vert\geq 2$. By Lemma~\ref{lem:cliques-ast}, one of the following two cases holds for $Q$.

\resetproofcases
\proofcase{$Q$ is a maximal clique of $G_i$ not contained in $X_i$, for some $i\in\{1,2\}$} Hence, neither $Q\cup\{v\}$ nor $Q\cup\{v'\}$ is a clique of $H_i^+$. Therefore, $Q$ is a maximal clique of $H_i^+$. Since $Q\cap W=Q\cap W_i$, the properness of $\phi_i$ implies that $Q\cap W_i$ is not monochromatic under $\phi_i$. Since $\phi$ agrees with $\phi_i$ on $Q\cap W_i$, $Q\cap W$ is not monochromatic under $\phi$.

\proofcase{$Q=R_1\cup R_2$, where $R_i$ is a maximal clique of $G_i[X_i]$ for each $i\in\{1,2\}$} We consider two subcases according to the sizes of $R_1\cap W$ and $R_2\cap W$.

\proofsubcase{$\vert R_i\cap W\vert\geq 2$ for some $i\in\{1,2\}$} The set $R_i\cup\{v'\}$ is a maximal clique of $H_i^+$, since any vertex of $H_i^+$ complete to $R_i\cup\{v'\}$ would belong to $X_i-R_i$, contradicting the maximality of $R_i$ in $G_i[X_i]$. Since $(R_i\cup\{v'\})\cap W_i=R_i\cap W$, the properness of $\phi_i$ implies that $R_i\cap W$ is not monochromatic under $\phi_i$. Since $\phi$ agrees with $\phi_i$ on $R_i\cap W$, $Q\cap W$ is not monochromatic under $\phi$.

\proofsubcase{$\vert R_i\cap W\vert\leq 1$ for each $i\in\{1,2\}$} Since $Q\cap W=(R_1\cap W)\cup(R_2\cap W)$ and $\vert Q\cap W\vert\geq 2$, both $R_1\cap W$ and $R_2\cap W$ have exactly one vertex. Let $x_i$ be the unique vertex of $R_i\cap W$ for each $i\in\{1,2\}$. For each $i\in\{1,2\}$, the set $R_i\cup\{v\}$ is a maximal clique of $H_i^+$, since any vertex of $H_i^+$ complete to $R_i\cup\{v\}$ would belong to $X_i-R_i$, contradicting the maximality of $R_i$ in $G_i[X_i]$. Since $(R_i\cup\{v\})\cap W_i=\{x_i,v\}$, the properness of $\phi_i$ implies $\phi_i(x_i)\neq\phi_i(v)$ for each $i\in\{1,2\}$. Since there are only two colors and $\phi_1(v)\neq\phi_2(v)$, it follows that $\phi(x_1)\neq\phi(x_2)$. Thus, $Q\cap W$ is not monochromatic under $\phi$.

\proofclaimspace The two cases for $Q$ above exhaust all possibilities and in each case $Q\cap W$ is not monochromatic under $\phi$. Since $Q$ is an arbitrary maximal clique of $G$ such that $\vert Q\cap W\vert\geq 2$, $\phi$ is a proper $2$-coloring of $W$ with respect to $G$. Since $W$ is an arbitrary subset of $V(G)$, Lemma~\ref{lemma:proper2coloring} implies that $G$ is balanced. This completes the proof of the `if' implication of assertion~\eqref{item:TSareNotCliques} and thus the proof of the proposition.\end{proof}

\begin{remark} Proposition~\ref{prop:ast-bal} shows that the balancedness of $H_1\ast H_2$ is not determined solely by the balancedness of $H_1$ and $H_2$. When one or both neighborhoods $N_{H_i}(v)$ are not cliques, it also depends on the balancedness of the corresponding graphs $H_i^+$ arising by adding a false twin of $v$. Consequently, any attempt to use a split decomposition tree to determine balancedness in general would have to keep track of these one-vertex augmentations $H_i^+$, rather than only the balancedness of $H_1$ and $H_2$.\end{remark}

In the inductive proof of Theorem~\ref{thm:DHbalancedGraphs}, applying Proposition~\ref{prop:ast-bal} requires balancedness information for the augmentations of the graphs $G_i=\DHG(T_i)$ obtained by adding a vertex $v$ with neighborhood $\TS(T_i)$ and, when needed, by adding a false twin of $v$. The following proposition provides the assertions about the balancedness of these augmentations needed in that proof.

\begin{proposition}\label{prop:TS-bal} Let $T$ be an $\OVE$-tree. Let $G=\DHG(T)$, let $H$ be the graph that arises from $G$ by adding a vertex $v$ such that $N_{H}(v)=\TS(T)$, and let $H^+$ be the graph that arises from $H$ by adding a false twin $v'$ of $v$. If $G$ is balanced, then both the following assertions hold:
\begin{enumerate}[(i)]
\item\label{item:2vertices} if no induced $\overline{3K_2}$ of $H^+$ contains both $v$ and $v'$, then $H^+$ is balanced; and
\item \label{item:1vertex} if no induced $\overline{3K_2}$ of $H$ contains $v$, then $H$ is balanced.
\end{enumerate}
\end{proposition}
\begin{proof} Let $T$ be an $\OVE$-tree and let $G$, $H$, and $H^+$ be as in the statement. We prove by induction on the definition of $\OVE$-trees that if $G$ is balanced, then both assertions~\eqref{item:2vertices} and~\eqref{item:1vertex} hold. The two assertions are proved simultaneously: in the inductive step, assertion~\eqref{item:2vertices} for $T=T_1\oplus_\alpha T_2$ is proved using assertion~\eqref{item:2vertices} for $T_1$ and $T_2$, while assertion~\eqref{item:1vertex} for $T$ is proved using, as needed, assertions~\eqref{item:2vertices} and~\eqref{item:1vertex} for $T_1$ and $T_2$, and assertion~\eqref{item:2vertices} already proved for the same tree $T$.

If $T$ has exactly one vertex, then $G=K_1$, $H=K_2$, and $H^+=P_3$. As $H$ and $H^+$ are balanced, both assertions~\eqref{item:2vertices} and~\eqref{item:1vertex} hold.

For the rest of the proof, suppose that $T=T_1\oplus_\alpha T_2$ with $\alpha\in\{\flabel,\tlabel,\plabel\}$ and that $G$ is balanced. For each $i\in\{1,2\}$, let $G_i=\DHG(T_i)$ and let $X_i=\TS(T_i)$. For each $i\in\{1,2\}$, let $H_i$ be the graph that arises from $G_i$ by adding a vertex $v$ such that $N_{H_i}(v)=X_i$ and let $H_i^+$ be the graph that arises from $H_i$ by adding a false twin $v'$ of $v$. By Remark~\ref{rmk:TS-path}, for each $i\in\{1,2\}$, $G_i=G[V(T_i)]$ and thus $G_i$ is balanced as an induced subgraph of the balanced graph $G$. By the induction hypothesis, for each $i\in\{1,2\}$, both assertions~\eqref{item:2vertices} and~\eqref{item:1vertex} hold for $T_i$. Notice that, if $\alpha\in\{\tlabel,\plabel\}$, then, by Definition~\ref{def:OVE-graph}, $G=H_1\ast H_2$.

\resetproofclaims
\resetproofcases \medskip
\noindent\emph{Proof of assertion~\eqref{item:2vertices}.} Assume that no induced $\overline{3K_2}$ of $H^+$ contains both $v$ and $v'$. We will prove that $H^+$ is balanced. By Lemma~\ref{lemma:WconVynoV'}, it suffices to show that every subset $W$ of $V(H^+)$ such that $v\in W$ and $v'\notin W$ admits a proper $2$-coloring with respect to $H^+$. Let $W$ be such a subset. For each $i\in\{1,2\}$, let $W_i=W\cap V(H_i)$. Since $v\in W$, for each $i\in\{1,2\}$, we have $W_i=(W\cap V(G_i))\cup\{v\}$. Since $v'\notin W$, we also have $W_i=W\cap V(H_i^+)$ for each $i\in\{1,2\}$. In the cases and subcases below, we will define $\phi$ by prescribing its values either on $W_1$ and $W_2-\{v\}$, or on $W_1-\{v\}$ and $W_2$. Since $W_1\cup W_2=W$ and $W_1\cap W_2=\{v\}$, each such $2$-coloring $\phi$ will be well defined.

We first prove a claim that will be used in the case analysis.

\proofclaim{For every $i\in\{1,2\}$, except possibly when both $\alpha=\plabel$ and $i=2$, $H_i^+$ is balanced}\label{claim:Hi+-balanced} Let $i\in\{1,2\}$ and suppose that $(\alpha,i)\neq(\plabel,2)$. Thus, by Definition~\ref{def:OVE-graph}, we have $H_i^+=H^+[V(G_i)\cup\{v,v'\}]$. Since $H_i^+$ is an induced subgraph of $H^+$, the assumption on $H^+$ implies that no induced $\overline{3K_2}$ of $H_i^+$ contains both $v$ and $v'$. Hence, by assertion~\eqref{item:2vertices} for $T_i$, $H_i^+$ is balanced.

\proofclaimspace We first consider three cases according to the value of $\alpha$.

\proofcase{$\alpha=\flabel$} By Claim~\ref{claim:Hi+-balanced}, $H_i^+$ is balanced for each $i\in\{1,2\}$. Hence, Lemma~\ref{lemma:proper2coloring} yields, for each $i\in\{1,2\}$, a proper $2$-coloring $\phi_i$ of $W_i$ with respect to $H_i^+$. Interchanging the colors of $\phi_2$ if necessary, we assume that $\phi_1(v)=\phi_2(v)$. Let $\phi$ be the $2$-coloring of $W$ such that $\phi\vert_{W_1}=\phi_1$ and $\phi\vert_{W_2-\{v\}}=\phi_2\vert_{W_2-\{v\}}$. Since $\phi_1(v)=\phi_2(v)$, $\phi$ agrees with $\phi_i$ on $W_i$ for each $i\in\{1,2\}$. Let $Q$ be a maximal clique of $H^+$ such that $\vert Q\cap W\vert\geq 2$. In this case, there are no edges of $H^+$ with one endpoint in $V(G_1)$ and the other in $V(G_2)$. Hence, $Q\subseteq V(H_i^+)$ for some $i\in\{1,2\}$. As $H_i^+$ is an induced subgraph of $H^+$, $Q$ is maximal in $H_i^+$. Since $Q\cap W=Q\cap W_i$ and $\vert Q\cap W\vert\geq 2$, the properness of $\phi_i$ implies that $Q\cap W_i$ is not monochromatic under $\phi_i$. Since $\phi$ agrees with $\phi_i$ on $W_i$, $Q\cap W$ is not monochromatic under $\phi$. Since $Q$ is an arbitrary maximal clique of $H^+$ such that $\vert Q\cap W\vert\geq 2$, $\phi$ is a proper $2$-coloring of $W$ with respect to $H^+$.

\proofcase{$\alpha=\tlabel$} By Claim~\ref{claim:Hi+-balanced}, $H_i^+$ is balanced for each $i\in\{1,2\}$. Hence, Lemma~\ref{lemma:proper2coloring} yields, for each $i\in\{1,2\}$, a proper $2$-coloring $\phi_i$ of $W_i$ with respect to $H_i^+$.

We first show that $X_1$ is a clique of $G_1$ or $X_2$ is a clique of $G_2$. Suppose, for a contradiction, that neither $X_1$ is a clique of $G_1$ nor $X_2$ is a clique of $G_2$. Let $x_1$ and $y_1$ be two vertices of $X_1$ that are nonadjacent in $G_1$ and let $x_2$ and $y_2$ be two vertices of $X_2$ that are nonadjacent in $G_2$. Since $\alpha=\tlabel$, $X_1$ is complete to $X_2$ in $G$. Therefore, $\{x_1,y_1,x_2,y_2,v,v'\}$ induces $\overline{3K_2}$ in $H^+$. This contradicts the assumption that no induced $\overline{3K_2}$ of $H^+$ contains both $v$ and $v'$.

By symmetry, suppose that $X_1$ is a clique of $G_1$. Hence, $X_1\cup\{v\}$ and $X_1\cup\{v'\}$ are maximal cliques of $H_1^+$, because each of $v$ and $v'$ is adjacent exactly to $X_1$ in $H_1^+$ and $v$ and $v'$ are nonadjacent in $H_1^+$.

We next prove a claim about maximal cliques contained in neither $V(G_1)$ nor $V(G_2)$ that will be used in the subcases below.

\proofclaim{If $Q$ is a maximal clique of $H^+$ contained in neither $V(G_1)$ nor $V(G_2)$, then $Q=X_1\cup R_2\cup\{u\}$, where $R_2$ is a maximal clique of $G_2[X_2]$ and $u\in\{v,v'\}$. In particular, $Q-X_1$ equals $R_2\cup\{u\}$, which is a maximal clique of $H_2^+$}\label{claim:aux-cliques} Let $Q$ be a maximal clique of $H^+$ contained in neither $V(G_1)$ nor $V(G_2)$. Since $v$ and $v'$ are nonadjacent in $H^+$, $Q$ contains at most one of them. We first show that $Q$ contains one of them. Suppose otherwise. Thus, $Q\subseteq V(G)$ and, since $G$ is an induced subgraph of $H^+$, $Q$ is a maximal clique of $G$. Since $Q$ is contained in neither $V(G_1)$ nor $V(G_2)$, Lemma~\ref{lem:cliques-ast} implies that $Q\subseteq X_1\cup X_2$. Hence, $Q\cup\{v\}$ is a clique of $H^+$, contradicting the maximality of $Q$ in $H^+$. Thus, $Q$ contains exactly one of $v$ and $v'$. Let $u$ be this vertex. Moreover, $(Q-\{u\})\cap V(G_i)\neq\emptyset$ for each $i\in\{1,2\}$. Indeed, if $(Q-\{u\})\cap V(G_i)=\emptyset$ for some $i\in\{1,2\}$, then $Q-\{u\}\subseteq X_{3-i}$ and hence any vertex of $X_i$ would be complete to $Q$ in $H^+$, contradicting the maximality of $Q$ in $H^+$. Therefore, $Q-\{u\}$ is a maximal clique of $G$, since any vertex of $G$ complete to $Q-\{u\}$ in $G$ would have to lie in $X_1\cup X_2$ (because $Q-\{u\}$ meets both $V(G_1)$ and $V(G_2)$) and would also be adjacent to $u$ in $H^+$, contradicting the maximality of $Q$ in $H^+$. Since $Q-\{u\}$ is contained in neither $V(G_1)$ nor $V(G_2)$ and $X_1$ is a clique of $G_1$, Lemma~\ref{lem:cliques-ast} implies that $Q-\{u\}=X_1\cup R_2$, where $R_2$ is a maximal clique of $G_2[X_2]$. Hence, $Q=X_1\cup R_2\cup\{u\}$. Finally, $Q-X_1$ equals $R_2\cup\{u\}$, which is a maximal clique of $H_2^+$ because $N_{H_2^+}(u)=X_2$. This proves the claim.

\proofclaimspace We now consider three subcases according to the size of $X_1\cap W$.

\proofsubcase{$\vert X_1\cap W\vert\geq 2$} Let $\phi$ be the $2$-coloring of $W$ such that $\phi\vert_{W_1}=\phi_1$ and $\phi\vert_{W_2-\{v\}}=\phi_2\vert_{W_2-\{v\}}$. Let $Q$ be a maximal clique of $H^+$ such that $\vert Q\cap W\vert\geq 2$. Suppose first that $Q\subseteq V(G_i)$ for some $i\in\{1,2\}$. Thus, $Q$ is a maximal clique of $H_i^+$, since $H_i^+=H^+[V(G_i)\cup\{v,v'\}]$. Since $Q\cap W=Q\cap W_i$, the properness of $\phi_i$ implies that $Q\cap W$ is not monochromatic under $\phi_i$. As $\phi$ agrees with $\phi_i$ on $Q\cap W$, $Q\cap W$ is not monochromatic under $\phi$. It only remains to consider the case where $Q$ is contained in neither $V(G_1)$ nor $V(G_2)$. Thus, by Claim~\ref{claim:aux-cliques}, we have $X_1\subseteq Q$. Since $X_1\cup\{v'\}$ is a maximal clique of $H_1^+$ and $(X_1\cup\{v'\})\cap W_1=X_1\cap W$, the properness of $\phi_1$ implies that $X_1\cap W$ is not monochromatic under $\phi_1$. Since $X_1\cap W\subseteq Q\cap W$ and $\phi$ agrees with $\phi_1$ on $W_1$, $Q\cap W$ is not monochromatic under $\phi$. Since $Q$ is an arbitrary maximal clique of $H^+$ such that $\vert Q\cap W\vert\geq 2$, $\phi$ is a proper $2$-coloring of $W$ with respect to $H^+$.

\proofsubcase{$\vert X_1\cap W\vert=1$} Let $w_1$ be the unique vertex of $X_1\cap W$. Since $X_1\cup\{v\}$ is a maximal clique of $H_1^+$, the properness of $\phi_1$ applied to $(X_1\cup\{v\})\cap W_1=\{w_1,v\}$ implies $\phi_1(w_1)\neq\phi_1(v)$. Interchanging the colors of $\phi_2$ if necessary, we assume that $\phi_2(v)=\phi_1(w_1)$. Let $\phi$ be the $2$-coloring of $W$ such that $\phi\vert_{W_1}=\phi_1$ and $\phi\vert_{W_2-\{v\}}=\phi_2\vert_{W_2-\{v\}}$. Let $Q$ be a maximal clique of $H^+$ such that $\vert Q\cap W\vert\geq 2$. Suppose first that $Q\subseteq V(G_i)$ for some $i\in\{1,2\}$. Thus, $Q$ is a maximal clique of $H_i^+$, since $H_i^+=H^+[V(G_i)\cup\{v,v'\}]$. Since $Q\cap W=Q\cap W_i$, the properness of $\phi_i$ implies that $Q\cap W$ is not monochromatic under $\phi_i$. As $\phi$ agrees with $\phi_i$ on $Q\cap W$, $Q\cap W$ is not monochromatic under $\phi$. It only remains to consider the case where $Q$ is contained in neither $V(G_1)$ nor $V(G_2)$. Thus, by Claim~\ref{claim:aux-cliques}, there exist a maximal clique $R_2$ of $G_2[X_2]$ and a vertex $u\in\{v,v'\}$ such that $Q=X_1\cup R_2\cup\{u\}$ and $R_2\cup\{u\}$ is a maximal clique of $H_2^+$. Since $X_1\cap W=\{w_1\}$, we have $w_1\in Q\cap W$. If $u=v$, then $\phi(v)=\phi_1(v)\neq\phi_1(w_1)=\phi(w_1)$, so $Q\cap W$ is not monochromatic under $\phi$. It only remains to consider the case where $u=v'$. Let $Q'=R_2\cup\{v\}$. Since $R_2\cup\{v'\}$ is a maximal clique of $H_2^+$ and $v$ and $v'$ are false twins in $H_2^+$, $Q'$ is also a maximal clique of $H_2^+$. Since $\vert Q\cap W\vert\geq 2$, $X_1\cap W=\{w_1\}$, and $v'\notin W$, $R_2\cap W$ is nonempty. Since $v\in W$, we have $\vert Q'\cap W\vert\geq 2$. Thus, the properness of $\phi_2$ ensures the existence of a vertex $w_2\in R_2\cap W$ such that $\phi_2(w_2)\neq\phi_2(v)$. Hence, $w_2\in Q\cap W$ and $\phi(w_2)=\phi_2(w_2)\neq\phi_2(v)=\phi_1(w_1)=\phi(w_1)$. Thus, $Q\cap W$ is not monochromatic under $\phi$. Since $Q$ is an arbitrary maximal clique of $H^+$ such that $\vert Q\cap W\vert\geq 2$, $\phi$ is a proper $2$-coloring of $W$ with respect to $H^+$.

\proofsubcase{$X_1\cap W=\emptyset$} Let $\phi$ be the $2$-coloring of $W$ such that $\phi\vert_{W_1-\{v\}}=\phi_1\vert_{W_1-\{v\}}$ and $\phi\vert_{W_2}=\phi_2$. Let $Q$ be a maximal clique of $H^+$ such that $\vert Q\cap W\vert\geq 2$. Suppose first that $Q\subseteq V(G_i)$ for some $i\in\{1,2\}$. Thus, $Q$ is a maximal clique of $H_i^+$, since $H_i^+=H^+[V(G_i)\cup\{v,v'\}]$. Since $Q\cap W=Q\cap W_i$, the properness of $\phi_i$ implies that $Q\cap W$ is not monochromatic under $\phi_i$. As $\phi$ agrees with $\phi_i$ on $Q\cap W$, $Q\cap W$ is not monochromatic under $\phi$. It only remains to consider the case where $Q$ is contained in neither $V(G_1)$ nor $V(G_2)$. Thus, by Claim~\ref{claim:aux-cliques}, $Q-X_1$ is a maximal clique of $H_2^+$. Since $X_1\cap W=\emptyset$, we have $Q\cap W=(Q-X_1)\cap W_2$. Since $\vert Q\cap W\vert\geq 2$, the properness of $\phi_2$ implies that $(Q-X_1)\cap W_2$ is not monochromatic under $\phi_2$. Since $\phi$ agrees with $\phi_2$ on $W_2$, $Q\cap W$ is not monochromatic under $\phi$. Since $Q$ is an arbitrary maximal clique of $H^+$ such that $\vert Q\cap W\vert\geq 2$, $\phi$ is a proper $2$-coloring of $W$ with respect to $H^+$.

\proofclaimspace The three subcases above cover all possibilities and, in each of them, $W$ admits a proper $2$-coloring with respect to $H^+$.

\proofcase{$\alpha=\plabel$} By Claim~\ref{claim:Hi+-balanced}, $H_1^+$ is balanced. Hence, Lemma~\ref{lemma:proper2coloring} yields a proper $2$-coloring $\phi_1$ of $W_1$ with respect to $H_1^+$.

We first consider two subcases according to whether or not $G_2[X_2]$ contains an induced $C_4$.

\proofsubcase{$G_2[X_2]$ contains no induced $C_4$} Notice that, if some induced $\overline{3K_2}$ of $H_2^+$ contained both $v$ and $v'$, then the four vertices other than $v$ and $v'$ would have to lie in $X_2$ and would induce a $C_4$ in $G_2[X_2]$, a contradiction. Thus, no induced $\overline{3K_2}$ of $H_2^+$ contains both $v$ and $v'$. By assertion~\eqref{item:2vertices} for $T_2$, $H_2^+$ is balanced. Hence, Lemma~\ref{lemma:proper2coloring} yields a proper $2$-coloring $\phi_2$ of $W_2$ with respect to $H_2^+$. Interchanging the colors of $\phi_2$ if necessary, we assume that $\phi_1(v)\neq\phi_2(v)$. Let $\phi$ be the $2$-coloring of $W$ such that $\phi\vert_{W_1}=\phi_1$ and $\phi\vert_{W_2-\{v\}}=\phi_2\vert_{W_2-\{v\}}$.

Let $Q$ be a maximal clique of $H^+$ such that $\vert Q\cap W\vert\geq 2$. Suppose first that $Q\subseteq V(H_1^+)$. Thus, $Q$ is a maximal clique of $H_1^+$, since $H_1^+$ is an induced subgraph of $H^+$. Since $Q\cap W=Q\cap W_1$ and $\vert Q\cap W\vert\geq 2$, the properness of $\phi_1$ implies that $Q\cap W_1$ is not monochromatic under $\phi_1$. Since $\phi$ agrees with $\phi_1$ on $W_1$, $Q\cap W$ is not monochromatic under $\phi$. Suppose now that $Q\subseteq V(G_2)$. Thus, $Q$ is a maximal clique of $G_2$. In fact, $Q$ is a maximal clique of $H_2^+$, since otherwise one of $v$ and $v'$ would be complete to $Q$, forcing $Q\subseteq X_2$ and thus $Q\cup\{x_1\}$ would be a clique of $H^+$ for any $x_1\in X_1$, contradicting the maximality of $Q$ in $H^+$. Since $Q\cap W=Q\cap(W_2-\{v\})$ and $\vert Q\cap W\vert\geq 2$, the properness of $\phi_2$ implies that $Q\cap W$ is not monochromatic under $\phi_2$. Since $\phi$ agrees with $\phi_2$ on $W_2-\{v\}$, $Q\cap W$ is not monochromatic under $\phi$.

It only remains to consider the case where $Q$ is contained in neither $V(H_1^+)$ nor $V(G_2)$. Since $V(H_1^+)$ and $V(G_2)$ partition $V(H^+)$, $Q$ meets both $V(H_1^+)$ and $V(G_2)$. As neither $v$ nor $v'$ is adjacent in $H^+$ to a vertex of $G_2$, $Q$ meets both $V(G_1)$ and $V(G_2)$. Since $\alpha=\plabel$, the only edges of $H^+$ with one endpoint in $V(G_1)$ and the other in $V(G_2)$ have one endpoint in $X_1$ and the other in $X_2$. Hence, $Q\subseteq X_1\cup X_2$. Thus, $Q\subseteq V(G)$. Since $G$ is an induced subgraph of $H^+$, $Q$ is a maximal clique of $G$. Since $Q$ meets both $V(G_1)$ and $V(G_2)$, Lemma~\ref{lem:cliques-ast} implies that $Q=R_1\cup R_2$, where $R_i$ is a maximal clique of $G_i[X_i]$ for each $i\in\{1,2\}$. If $\vert R_i\cap W\vert\geq 2$ for some $i\in\{1,2\}$, then $R_i\cup\{v'\}$ is a maximal clique of $H_i^+$, because $R_i$ is maximal in $G_i[X_i]$ and $v'$ is adjacent exactly to $X_i$ in $H_i^+$. Since $(R_i\cup\{v'\})\cap W_i=R_i\cap W$, the properness of $\phi_i$ implies that $R_i\cap W$ is not monochromatic under $\phi_i$. Since $R_i\cap W\subseteq Q\cap W$ and $\phi$ agrees with $\phi_i$ on $W_i-\{v\}$, $Q\cap W$ is not monochromatic under $\phi$. It only remains to consider the case where $\vert R_i\cap W\vert=1$ for each $i\in\{1,2\}$. Let $w_i$ be the unique vertex of $R_i\cap W$ for each $i\in\{1,2\}$. Moreover, for each $i\in\{1,2\}$, $R_i\cup\{v\}$ is a maximal clique of $H_i^+$ and $(R_i\cup\{v\})\cap W_i=\{v,w_i\}$. Hence, $\phi_i(v)\neq\phi_i(w_i)$ for each $i\in\{1,2\}$. Since $\phi_1(v)\neq\phi_2(v)$ and there are only two colors, it follows that $\phi(w_1)\neq\phi(w_2)$. Thus, $Q\cap W$ is not monochromatic under $\phi$. Since $Q$ is an arbitrary maximal clique of $H^+$ such that $\vert Q\cap W\vert\geq 2$, $\phi$ is a proper $2$-coloring of $W$ with respect to $H^+$.

\proofsubcase{$G_2[X_2]$ contains an induced $C_4$} We first prove that $X_1$ is a clique of $G_1$. Suppose, for a contradiction, that $X_1$ contains two vertices $x$ and $x'$ that are nonadjacent in $G_1$. Let $C$ be a subset of $X_2$ that induces a $C_4$ in $G_2$. Since $\alpha=\plabel$, every vertex of $X_1$ is adjacent in $G$ to every vertex of $X_2$. Hence, $C\cup\{x,x'\}$ induces $\overline{3K_2}$ in $G$, contradicting the balancedness of $G$. Thus, $X_1$ is a clique of $G_1$. Hence, $X_1\cup\{v\}$ and $X_1\cup\{v'\}$ are maximal cliques of $H_1^+$, because each of $v$ and $v'$ is adjacent exactly to $X_1$ in $H_1^+$ and $v$ and $v'$ are nonadjacent in $H_1^+$.

Since $X_1$ is nonempty, let $x$ be a vertex of $X_1$. Since $\alpha=\plabel$, the graph $H_2$ is isomorphic to $G[V(G_2)\cup\{x\}]$. Thus, as $G$ is balanced, $H_2$ is balanced. Hence, Lemma~\ref{lemma:proper2coloring} yields a proper $2$-coloring $\phi_2$ of $W_2$ with respect to $H_2$. Since $H_2$ is balanced, its clique hypergraph $\mathcal{C}(H_2)$ is balanced. Theorem~\ref{thm:balancedhypergraphs} yields a proper $2$-coloring $\psi_2$ of the induced subhypergraph $\mathcal{C}(H_2)_{W_2-\{v\}}$.

We prove three claims that will be used in the subcases below.

\proofclaim{If $Q$ is a maximal clique of $H_2$ contained in $V(G_2)$ and $\vert Q\cap W\vert\geq 2$, then $Q\cap W$ is not monochromatic under $\psi_2$}\label{claim:psi2-G2} Indeed, since $W_2-\{v\}=W\cap V(G_2)$ and $Q\subseteq V(G_2)$, we have $Q\cap W=Q\cap(W_2-\{v\})$. Hence, $\vert Q\cap(W_2-\{v\})\vert\geq 2$ and, by the choice of $\psi_2$, $Q\cap(W_2-\{v\})$ is not monochromatic under $\psi_2$. Therefore, $Q\cap W$ is not monochromatic under $\psi_2$. This proves the claim.

\proofclaim{If $Q$ is a maximal clique of $H^+$ and $Q\subseteq V(G_2)$, then $Q$ is a maximal clique of $H_2$}\label{claim:pendant-H2} Since $G_2$ is an induced subgraph of $H^+$, $Q$ is a maximal clique of $G_2$. Thus, if $Q$ were not maximal in $H_2$, then $v$ would be complete to $Q$ in $H_2$ and hence $Q\subseteq X_2$. Since $X_1$ is nonempty and complete to $X_2$ in $H^+$, any vertex of $X_1$ would be complete to $Q$ in $H^+$, contradicting the maximality of $Q$ in $H^+$. Thus, $Q$ is a maximal clique of $H_2$. This proves the claim.

\proofclaim{If $Q$ is a maximal clique of $H^+$ contained in neither $V(H_1^+)$ nor $V(G_2)$, then $Q=X_1\cup R_2$, where $R_2$ is a maximal clique of $G_2[X_2]$ and $R_2\cup\{v\}$ is a maximal clique of $H_2$}\label{claim:pendant-mixed} Since $V(H_1^+)$ and $V(G_2)$ partition $V(H^+)$, $Q$ meets both $V(H_1^+)$ and $V(G_2)$. As neither $v$ nor $v'$ is adjacent in $H^+$ to a vertex of $G_2$, $Q$ meets both $V(G_1)$ and $V(G_2)$ and is contained in $V(G)$. Since $G$ is an induced subgraph of $H^+$, $Q$ is a maximal clique of $G$. Since $X_1$ is a clique of $G_1$, Lemma~\ref{lem:cliques-ast} implies that $Q=X_1\cup R_2$, where $R_2$ is a maximal clique of $G_2[X_2]$. Since $N_{H_2}(v)=X_2$, the maximality of $R_2$ in $G_2[X_2]$ implies that $R_2\cup\{v\}$ is a maximal clique of $H_2$. This proves the claim.

\proofclaimspace We now consider three subcases according to the size of $X_1\cap W$.

\proofsubsubcase{$\vert X_1\cap W\vert\geq 2$} Let $\phi$ be the $2$-coloring of $W$ such that $\phi\vert_{W_1}=\phi_1$ and $\phi\vert_{W_2-\{v\}}=\psi_2$. Let $Q$ be a maximal clique of $H^+$ such that $\vert Q\cap W\vert\geq 2$. Suppose first that $Q\subseteq V(H_1^+)$. Since $H_1^+$ is an induced subgraph of $H^+$, $Q$ is a maximal clique of $H_1^+$. Since $Q\cap W=Q\cap W_1$, the properness of $\phi_1$ implies that $Q\cap W$ is not monochromatic under $\phi_1$ and hence not monochromatic under $\phi$. Suppose now that $Q\subseteq V(G_2)$. By Claim~\ref{claim:pendant-H2}, $Q$ is a maximal clique of $H_2$. By Claim~\ref{claim:psi2-G2}, $Q\cap W$ is not monochromatic under $\psi_2$. Since $Q\cap W\subseteq W_2-\{v\}$ and $\phi$ agrees with $\psi_2$ on $W_2-\{v\}$, $Q\cap W$ is not monochromatic under $\phi$. It only remains to consider the case where $Q$ is contained in neither $V(H_1^+)$ nor $V(G_2)$. Thus, Claim~\ref{claim:pendant-mixed} implies that $X_1\subseteq Q$. Since $X_1\cup\{v'\}$ is a maximal clique of $H_1^+$ and $(X_1\cup\{v'\})\cap W_1=X_1\cap W$, the properness of $\phi_1$ implies that $X_1\cap W$ is not monochromatic under $\phi_1$. Since $X_1\cap W\subseteq Q\cap W$, $Q\cap W$ is not monochromatic under $\phi$. Since $Q$ is an arbitrary maximal clique of $H^+$ such that $\vert Q\cap W\vert\geq 2$, $\phi$ is a proper $2$-coloring of $W$ with respect to $H^+$.

\proofsubsubcase{$\vert X_1\cap W\vert=1$} Let $w_1$ be the unique vertex of $X_1\cap W$. Since $X_1\cup\{v\}$ is a maximal clique of $H_1^+$, the properness of $\phi_1$ applied to $(X_1\cup\{v\})\cap W_1=\{w_1,v\}$ implies $\phi_1(w_1)\neq\phi_1(v)$. Interchanging the colors of $\phi_2$ if necessary, we assume that $\phi_2(v)=\phi_1(w_1)$. Let $\phi$ be the $2$-coloring of $W$ such that $\phi\vert_{W_1}=\phi_1$ and $\phi\vert_{W_2-\{v\}}=\phi_2\vert_{W_2-\{v\}}$. Let $Q$ be a maximal clique of $H^+$ such that $\vert Q\cap W\vert\geq 2$. Suppose first that $Q\subseteq V(H_1^+)$. Since $H_1^+$ is an induced subgraph of $H^+$, $Q$ is a maximal clique of $H_1^+$. Since $Q\cap W=Q\cap W_1$, the properness of $\phi_1$ implies that $Q\cap W$ is not monochromatic under $\phi_1$ and hence not monochromatic under $\phi$. Suppose now that $Q\subseteq V(G_2)$. By Claim~\ref{claim:pendant-H2}, $Q$ is a maximal clique of $H_2$. Since $Q\cap W=Q\cap W_2$, the properness of $\phi_2$ implies that $Q\cap W$ is not monochromatic under $\phi_2$. Since $Q\cap W\subseteq W_2-\{v\}$ and $\phi$ agrees with $\phi_2$ on $W_2-\{v\}$, $Q\cap W$ is not monochromatic under $\phi$. It only remains to consider the case where $Q$ is contained in neither $V(H_1^+)$ nor $V(G_2)$. Thus, by Claim~\ref{claim:pendant-mixed}, $Q=X_1\cup R_2$ for some $R_2\subseteq X_2$ such that $R_2\cup\{v\}$ is a maximal clique of $H_2$. Since $\vert Q\cap W\vert\geq 2$ and $X_1\cap W=\{w_1\}$, $R_2\cap W$ is nonempty. Since $v\in W_2$, we have $\vert (R_2\cup\{v\})\cap W_2\vert\geq 2$. The properness of $\phi_2$ ensures the existence of $w_2\in R_2\cap W$ such that $\phi_2(w_2)\neq\phi_2(v)$. Hence, $w_2\in Q\cap W$ and $\phi(w_2)=\phi_2(w_2)\neq\phi_2(v)=\phi_1(w_1)=\phi(w_1)$. Thus, $Q\cap W$ is not monochromatic under $\phi$. Since $Q$ is an arbitrary maximal clique of $H^+$ such that $\vert Q\cap W\vert\geq 2$, $\phi$ is a proper $2$-coloring of $W$ with respect to $H^+$.

\proofsubsubcase{$X_1\cap W=\emptyset$} Let $\phi$ be the $2$-coloring of $W$ such that $\phi\vert_{W_1}=\phi_1$ and $\phi\vert_{W_2-\{v\}}=\psi_2$. Let $Q$ be a maximal clique of $H^+$ such that $\vert Q\cap W\vert\geq 2$. Suppose first that $Q\subseteq V(H_1^+)$. Since $H_1^+$ is an induced subgraph of $H^+$, $Q$ is a maximal clique of $H_1^+$. Since $Q\cap W=Q\cap W_1$, the properness of $\phi_1$ implies that $Q\cap W$ is not monochromatic under $\phi_1$ and hence not monochromatic under $\phi$. Suppose now that $Q\subseteq V(G_2)$. By Claim~\ref{claim:pendant-H2}, $Q$ is a maximal clique of $H_2$. By Claim~\ref{claim:psi2-G2}, $Q\cap W$ is not monochromatic under $\psi_2$. Since $Q\cap W\subseteq W_2-\{v\}$ and $\phi$ agrees with $\psi_2$ on $W_2-\{v\}$, $Q\cap W$ is not monochromatic under $\phi$. It only remains to consider the case where $Q$ is contained in neither $V(H_1^+)$ nor $V(G_2)$. Thus, by Claim~\ref{claim:pendant-mixed}, $Q=X_1\cup R_2$ for some $R_2\subseteq X_2$ such that $R_2\cup\{v\}$ is a maximal clique of $H_2$. Since $X_1\cap W=\emptyset$, we have $Q\cap W=(R_2\cup\{v\})\cap W\cap V(G_2)$. Since $\vert Q\cap W\vert\geq 2$ and $R_2\cup\{v\}$ is a maximal clique of $H_2$, the properness of $\psi_2$ implies that $Q\cap W$ is not monochromatic under $\psi_2$ and hence not monochromatic under $\phi$. Since $Q$ is an arbitrary maximal clique of $H^+$ such that $\vert Q\cap W\vert\geq 2$, $\phi$ is a proper $2$-coloring of $W$ with respect to $H^+$.

\proofclaimspace Subcases 3.1 and 3.2 cover all possibilities in the case $\alpha=\plabel$ and, in each of them, $W$ admits a proper $2$-coloring with respect to $H^+$. We have reached the same conclusion in the cases where $\alpha=\flabel$ or $\alpha=\tlabel$. Since $W$ is an arbitrary subset of $V(H^+)$ with $v\in W$ and $v'\notin W$, Lemma~\ref{lemma:WconVynoV'} implies that $H^+$ is balanced. This proves assertion~\eqref{item:2vertices}.

\medskip
\noindent\emph{Proof of assertion~\eqref{item:1vertex}.} Assume that no induced $\overline{3K_2}$ of $H$ contains $v$. We will prove that $H$ is balanced.
\resetproofcases

We again consider three cases according to the value of $\alpha$.

\proofcase{$\alpha=\flabel$} Thus, $H_i=H[V(G_i)\cup\{v\}]$ and no induced $\overline{3K_2}$ of $H_i$ contains $v$, for each $i\in\{1,2\}$. Since assertion~\eqref{item:1vertex} holds for $T_i$, $H_i$ is balanced for each $i\in\{1,2\}$.

Let $W$ be an arbitrary subset of $V(H)$. By Lemma~\ref{lemma:proper2coloring}, to prove that $H$ is balanced, it suffices to give a proper $2$-coloring of $W$ with respect to $H$.

For each $i\in\{1,2\}$, let $W_i=W\cap V(H_i)$. Since both $H_1$ and $H_2$ are balanced, Lemma~\ref{lemma:proper2coloring} yields, for each $i\in\{1,2\}$, a proper $2$-coloring $\phi_i$ of $W_i$ with respect to $H_i$. If $v\in W$, interchange the colors of $\phi_2$ if necessary so that $\phi_1(v)=\phi_2(v)$. Let $\phi$ be the $2$-coloring of $W$ such that $\phi\vert_{W_i}=\phi_i$ for each $i\in\{1,2\}$. This is well defined because $W_1\cup W_2=W$ and either $W_1\cap W_2=\emptyset$ or $W_1\cap W_2=\{v\}$ with the two colorings agreeing on $v$.

Let $Q$ be a maximal clique of $H$ such that $\vert Q\cap W\vert\geq 2$. In this case, there are no edges of $H$ with one endpoint in $V(G_1)$ and the other in $V(G_2)$. Hence, $Q\subseteq V(H_i)$ for some $i\in\{1,2\}$. Thus, as $H_i$ is an induced subgraph of $H$, $Q$ is a maximal clique of $H_i$. Since $Q\cap W=Q\cap W_i$ and $\vert Q\cap W_i\vert\geq 2$, by the properness of $\phi_i$, $Q\cap W$ is not monochromatic under $\phi$. Since $Q$ is an arbitrary maximal clique of $H$ such that $\vert Q\cap W\vert\geq 2$, $\phi$ is a proper $2$-coloring of $W$ with respect to $H$. Since $W$ is an arbitrary subset of $V(H)$, Lemma~\ref{lemma:proper2coloring} implies that $H$ is balanced in this case.

\proofcase{$\alpha=\tlabel$} Thus, $H_i=H[V(G_i)\cup\{v\}]$ and no induced $\overline{3K_2}$ of $H_i$ contains $v$, for each $i\in\{1,2\}$. Since assertion~\eqref{item:1vertex} holds for $T_i$, $H_i$ is balanced for each $i\in\{1,2\}$.

Let $W$ be an arbitrary subset of $V(H)$. By Lemma~\ref{lemma:proper2coloring}, to prove that $H$ is balanced, it suffices to give a proper $2$-coloring of $W$ with respect to $H$.

We prove three claims about maximal cliques of $H$ that will be used in the subcases below.

\proofclaim{If $Q$ is a maximal clique of $H$ and $Q\subseteq V(G_i)$ for some $i\in\{1,2\}$, then $Q$ is a maximal clique of both $H_i$ and $H_i^+$}\label{claim:contained-cliques} Thus, as $H_i$ is an induced subgraph of $H$, $Q$ is a maximal clique of $H_i$. Moreover, $Q$ is not contained in $X_i$, since otherwise $Q\cup\{v\}$ would be a clique of $H_i$, contradicting the maximality of $Q$ in $H_i$. As $v'$ is adjacent exactly to $X_i$ in $H_i^+$, the vertex $v'$ is not complete to $Q$. Hence, $Q$ is also a maximal clique of $H_i^+$. This proves the claim.

\proofclaim{If $Q$ is a maximal clique of $H$ contained in neither $V(G_1)$ nor $V(G_2)$, then $Q=R_1\cup R_2\cup\{v\}$, where $R_i$ is a maximal clique of $G_i[X_i]$ for each $i\in\{1,2\}$. Moreover, for each $i\in\{1,2\}$, both $R_i\cup\{v\}$ and $R_i\cup\{v'\}$ are maximal cliques of $H_i^+$. In particular, $R_i\cup\{v\}$ is a maximal clique of $H_i$ for each $i\in\{1,2\}$}\label{claim:mixed-core} For each $i\in\{1,2\}$, let $R_i=Q\cap V(G_i)$. We first prove that $R_i\neq\emptyset$ for each $i\in\{1,2\}$. By symmetry, suppose, for a contradiction, that $R_2=\emptyset$. Since $Q$ is not contained in $V(G_1)$, we have $v\in Q$. As $Q$ is a clique and $N_H(v)\cap V(G_1)=X_1$, it follows that $Q\subseteq X_1\cup\{v\}$. Since $X_2$ is nonempty and complete to $X_1\cup\{v\}$ in $H$, a vertex of $X_2$ is complete to $Q$ in $H$, contradicting the maximality of $Q$ in $H$. This contradiction proves that indeed $R_1$ and $R_2$ are nonempty.

For each $i\in\{1,2\}$, every vertex of $R_i$ has a neighbor outside $V(G_i)$ in $H$, because $R_{3-i}$ is nonempty and $Q$ is a clique. Hence, $R_i\subseteq X_i$ for each $i\in\{1,2\}$. If $v\notin Q$, then $Q=R_1\cup R_2$ and $Q\cup\{v\}$ is a clique of $H$, contradicting the maximality of $Q$ in $H$. Thus, $v\in Q$ and so $Q=R_1\cup R_2\cup\{v\}$. As $G$ is an induced subgraph of $H$, the set $Q-\{v\}$ is a clique of $G$. Moreover, $Q-\{v\}$ is maximal in $G$, since any vertex of $G$ complete to $Q-\{v\}$ in $G$ would belong to $X_1\cup X_2$ (because $Q-\{v\}$ meets both $V(G_1)$ and $V(G_2)$) and thus would be adjacent to $v$ in $H$, contradicting the maximality of $Q$ in $H$. Hence, by Lemma~\ref{lem:cliques-ast}, $R_i$ is a maximal clique of $G_i[X_i]$ for each $i\in\{1,2\}$. This proves the first assertion of the claim.

Let $i\in\{1,2\}$. Since $R_i\subseteq X_i$, both $R_i\cup\{v\}$ and $R_i\cup\{v'\}$ are cliques of $H_i^+$. They are maximal because $v$ and $v'$ are nonadjacent in $H_i^+$ and any vertex of $G_i$ complete to one of them would belong to $X_i-R_i$ and would be complete to $R_i$ in $G_i[X_i]$, contradicting the maximality of $R_i$ in $G_i[X_i]$. Since $H_i$ is an induced subgraph of $H_i^+$, it follows that $R_i\cup\{v\}$ is a maximal clique of $H_i$. Since $i\in\{1,2\}$ was arbitrary, this completes the proof of the claim.

\proofclaimspace We consider three subcases.

\proofsubcase{$v\in W$} For each $i\in\{1,2\}$, let $W_i=W\cap V(H_i)$. Since $H_i$ is balanced for each $i\in\{1,2\}$, Lemma~\ref{lemma:proper2coloring} yields, for each $i\in\{1,2\}$, a proper $2$-coloring $\phi_i$ of $W_i$ with respect to $H_i$. Interchanging the colors of $\phi_2$ if necessary, we assume that $\phi_1(v)=\phi_2(v)$. Let $\phi$ be the $2$-coloring of $W$ such that $\phi\vert_{W_i}=\phi_i$ for each $i\in\{1,2\}$. This is well defined because $W_1\cup W_2=W$, $W_1\cap W_2=\{v\}$, and the two colorings agree on $v$.

Let $Q$ be a maximal clique of $H$ such that $\vert Q\cap W\vert\geq 2$. Suppose first that $Q\subseteq V(G_i)$ for some $i\in\{1,2\}$. By Claim~\ref{claim:contained-cliques}, $Q$ is a maximal clique of $H_i$. Since $Q\cap W=Q\cap W_i$, the properness of $\phi_i$ implies that $Q\cap W$ is not monochromatic under $\phi_i$. As $\phi$ agrees with $\phi_i$ on $Q\cap W$, $Q\cap W$ is not monochromatic under $\phi$. It only remains to consider the case where $Q$ is contained in neither $V(G_1)$ nor $V(G_2)$. By Claim~\ref{claim:mixed-core}, $Q=R_1\cup R_2\cup\{v\}$, where $R_i\subseteq X_i$ and $R_i\cup\{v\}$ is a maximal clique of $H_i$ for each $i\in\{1,2\}$. Since $v\in W$ and $\vert Q\cap W\vert\geq 2$, there is some $i\in\{1,2\}$ such that $(R_i\cup\{v\})\cap W_i$ has at least two vertices. By the properness of $\phi_i$, $(R_i\cup\{v\})\cap W_i$ is not monochromatic under $\phi_i$. Since $\phi$ agrees with $\phi_i$ on $W_i$ and $(R_i\cup\{v\})\cap W_i\subseteq Q\cap W$, it follows that $Q\cap W$ is not monochromatic under $\phi$. Since $Q$ is an arbitrary maximal clique of $H$ such that $\vert Q\cap W\vert\geq 2$, $\phi$ is a proper $2$-coloring of $W$ with respect to $H$ in this subcase.

\proofsubcase{$v\notin W$ and neither $G_1[X_1]$ nor $G_2[X_2]$ contains an induced $C_4$} For each $i\in\{1,2\}$, no induced $\overline{3K_2}$ of $H_i^+$ contains both $v$ and $v'$, since otherwise the four remaining vertices would lie in $X_i$ and induce a $C_4$ in $G_i[X_i]$. Thus, by assertion~\eqref{item:2vertices} for $T_i$, $H_i^+$ is balanced for each $i\in\{1,2\}$.

For each $i\in\{1,2\}$, let $W_i=(W\cap V(G_i))\cup\{v\}$. Since $H_i^+$ is balanced for each $i\in\{1,2\}$, Lemma~\ref{lemma:proper2coloring} yields, for each $i\in\{1,2\}$, a proper $2$-coloring $\phi_i$ of $W_i$ with respect to $H_i^+$. Interchanging the colors of $\phi_2$ if necessary, we assume that $\phi_1(v)\neq\phi_2(v)$. Let $\phi$ be the $2$-coloring of $W$ such that $\phi\vert_{W\cap V(G_i)}=\phi_i\vert_{W\cap V(G_i)}$ for each $i\in\{1,2\}$. This is well defined because $W\subseteq V(G)$ and $V(G_1)$ and $V(G_2)$ partition $V(G)$.

Let $Q$ be a maximal clique of $H$ such that $\vert Q\cap W\vert\geq 2$. Suppose first that $Q\subseteq V(G_i)$ for some $i\in\{1,2\}$. By Claim~\ref{claim:contained-cliques}, $Q$ is a maximal clique of $H_i^+$. Since $Q\cap W=Q\cap W_i$, the properness of $\phi_i$ implies that $Q\cap W$ is not monochromatic under $\phi_i$. As $\phi$ agrees with $\phi_i$ on $Q\cap W$, $Q\cap W$ is not monochromatic under $\phi$. It only remains to consider the case where $Q$ is contained in neither $V(G_1)$ nor $V(G_2)$. By Claim~\ref{claim:mixed-core}, $Q=R_1\cup R_2\cup\{v\}$, where, for each $i\in\{1,2\}$, $R_i$ is a maximal clique of $G_i[X_i]$ and both $R_i\cup\{v\}$ and $R_i\cup\{v'\}$ are maximal cliques of $H_i^+$. Since $v\notin W$, it follows that $Q\cap W=(R_1\cap W)\cup(R_2\cap W)$. Suppose first that $\vert R_i\cap W\vert\geq 2$ for some $i\in\{1,2\}$. Since $(R_i\cup\{v'\})\cap W_i=R_i\cap W$, the properness of $\phi_i$ implies that $R_i\cap W$ is not monochromatic under $\phi_i$. Since $\phi$ agrees with $\phi_i$ on $R_i\cap W$ and $R_i\cap W\subseteq Q\cap W$, it follows that $Q\cap W$ is not monochromatic under $\phi$. It only remains to consider the case where $\vert R_i\cap W\vert\leq 1$ for each $i\in\{1,2\}$. Since $\vert Q\cap W\vert\geq 2$, we have $\vert R_i\cap W\vert=1$ for each $i\in\{1,2\}$. Let $w_i$ be the unique vertex of $R_i\cap W$. Since $(R_i\cup\{v\})\cap W_i=\{w_i,v\}$, the properness of $\phi_i$ implies that $\phi_i(w_i)\neq\phi_i(v)$ for each $i\in\{1,2\}$. As there are only two colors and $\phi_1(v)\neq\phi_2(v)$, it follows that $\phi(w_1)\neq\phi(w_2)$. Thus, $Q\cap W$ is not monochromatic under $\phi$. Since $Q$ is an arbitrary maximal clique of $H$ such that $\vert Q\cap W\vert\geq 2$, $\phi$ is a proper $2$-coloring of $W$ with respect to $H$ in this subcase.

\proofsubcase{$v\notin W$ and $G_i[X_i]$ contains an induced $C_4$ for some $i\in\{1,2\}$} By symmetry, suppose that $G_1[X_1]$ contains an induced $C_4$. We first show that $X_2$ is a clique of $G_2$. Suppose, for a contradiction, that $X_2$ contains two vertices $x_2$ and $y_2$ that are nonadjacent in $G_2$. Let $C$ be a subset of $X_1$ inducing a $C_4$ in $G_1$. Since $X_1$ is complete to $X_2$ in $G$, $C\cup\{x_2,y_2\}$ induces $\overline{3K_2}$ in $G$, contradicting the balancedness of $G$. Hence, $X_2$ is a clique of $G_2$. In particular, $G_2[X_2]$ contains no induced $C_4$. Therefore, by assertion~\eqref{item:2vertices} for $T_2$, $H_2^+$ is balanced.

Let $W_1=W\cap V(G_1)$, let $W_1^v=(W\cap V(G_1))\cup\{v\}$, and let $W_2=(W\cap V(G_2))\cup\{v\}$. Since $H_1$ and $H_2^+$ are balanced, Lemma~\ref{lemma:proper2coloring} yields a proper $2$-coloring $\phi_1$ of $W_1$ with respect to $H_1$, a proper $2$-coloring $\phi_1^v$ of $W_1^v$ with respect to $H_1$, and a proper $2$-coloring $\phi_2$ of $W_2$ with respect to $H_2^+$. If $\vert X_2\cap W\vert=1$, let $w_2$ be the unique vertex of $X_2\cap W$ and interchange the colors of $\phi_1^v$ if necessary so that $\phi_1^v(v)=\phi_2(w_2)$. Let $\phi$ be the $2$-coloring of $W$ defined as follows. If $\vert X_2\cap W\vert=1$, set $\phi\vert_{W\cap V(G_1)}=\phi_1^v\vert_{W\cap V(G_1)}$ and $\phi\vert_{W\cap V(G_2)}=\phi_2\vert_{W\cap V(G_2)}$. Otherwise, set $\phi\vert_{W\cap V(G_1)}=\phi_1$ and $\phi\vert_{W\cap V(G_2)}=\phi_2\vert_{W\cap V(G_2)}$.

Let $Q$ be a maximal clique of $H$ such that $\vert Q\cap W\vert\geq 2$. Suppose first that $Q\subseteq V(G_1)$. By Claim~\ref{claim:contained-cliques}, $Q$ is a maximal clique of $H_1$. Since $Q\subseteq V(G_1)$, we have $Q\cap W=Q\cap W_1=Q\cap W_1^v$. If $\vert X_2\cap W\vert=1$, then the properness of $\phi_1^v$ implies that $Q\cap W$ is not monochromatic under $\phi_1^v$ and, since $\phi$ agrees with $\phi_1^v$ on $Q\cap W$, $Q\cap W$ is not monochromatic under $\phi$. If $\vert X_2\cap W\vert\neq 1$, then the properness of $\phi_1$ implies that $Q\cap W$ is not monochromatic under $\phi_1$ and, since $\phi$ agrees with $\phi_1$ on $Q\cap W$, $Q\cap W$ is not monochromatic under $\phi$. Suppose now that $Q\subseteq V(G_2)$. By Claim~\ref{claim:contained-cliques}, $Q$ is a maximal clique of $H_2^+$. Since $Q\cap W=Q\cap W_2$, the properness of $\phi_2$ implies that $Q\cap W$ is not monochromatic under $\phi_2$. As $\phi$ agrees with $\phi_2$ on $Q\cap W$, $Q\cap W$ is not monochromatic under $\phi$. It only remains to consider the case where $Q$ is contained in neither $V(G_1)$ nor $V(G_2)$. Since $X_2$ is a clique of $G_2$, Claim~\ref{claim:mixed-core} implies that $Q=R_1\cup X_2\cup\{v\}$, where $R_1$ is a maximal clique of $G_1[X_1]$, $X_2\cup\{v'\}$ is a maximal clique of $H_2^+$, and $R_1\cup\{v\}$ is a maximal clique of $H_1$. Suppose first that $\vert X_2\cap W\vert\geq 2$. Since $(X_2\cup\{v'\})\cap W_2=X_2\cap W$, the properness of $\phi_2$ implies that $X_2\cap W$ is not monochromatic under $\phi_2$. Since $\phi$ agrees with $\phi_2$ on $X_2\cap W$ and $X_2\cap W\subseteq Q\cap W$, it follows that $Q\cap W$ is not monochromatic under $\phi$. Suppose next that $X_2\cap W=\emptyset$. In this case, $Q\cap W=R_1\cap W$. Since $(R_1\cup\{v\})\cap W_1=R_1\cap W=Q\cap W$, the properness of $\phi_1$ implies that $Q\cap W$ is not monochromatic under $\phi_1$. As $\phi$ agrees with $\phi_1$ on $Q\cap W$, $Q\cap W$ is not monochromatic under $\phi$. Finally, suppose that $\vert X_2\cap W\vert=1$ and let $w_2$ be the unique vertex of $X_2\cap W$. Since $\vert Q\cap W\vert\geq 2$ and $X_2\cap W=\{w_2\}$, $R_1\cap W$ is nonempty. Since $(R_1\cup\{v\})\cap W_1^v$ contains $v$ and at least one vertex of $R_1\cap W$, by the properness of $\phi_1^v$, there exists a vertex $w_1\in R_1\cap W$ such that $\phi_1^v(w_1)\neq\phi_1^v(v)=\phi_2(w_2)$. Hence, $Q\cap W$ is not monochromatic under $\phi$. Since $Q$ is an arbitrary maximal clique of $H$ such that $\vert Q\cap W\vert\geq 2$, $\phi$ is a proper $2$-coloring of $W$ with respect to $H$ in this subcase.

\proofclaimspace The three subcases above cover all possibilities and, in each of them, $W$ admits a proper $2$-coloring with respect to $H$. Since $W$ is an arbitrary subset of $V(H)$, Lemma~\ref{lemma:proper2coloring} implies that $H$ is balanced in this case.

\proofcase{$\alpha=\plabel$} We first prove that no induced $\overline{3K_2}$ of $H^+$ contains both $v$ and $v'$. Suppose, for a contradiction, that such an induced subgraph exists. Since $v$ and $v'$ are false twins with common neighborhood $\TS(T)=X_1$, the other four vertices lie in $X_1$ and induce a $C_4$ in $G_1[X_1]$. By Remark~\ref{rmk:TS-path}, $X_2$ is nonempty. Let $x_2$ be a vertex of $X_2$. Since $\alpha=\plabel$, $x_2$ is adjacent in $H$ to every vertex of $X_1$ and is nonadjacent to $v$ in $H$. Thus, replacing $v'$ by $x_2$ produces an induced $\overline{3K_2}$ of $H$ containing $v$, contradicting the assumption that no induced $\overline{3K_2}$ of $H$ contains $v$. Hence, no induced $\overline{3K_2}$ of $H^+$ contains both $v$ and $v'$. By assertion~\eqref{item:2vertices}, already proved for $T$, $H^+$ is balanced. Since $H=H^+-v'$ and the class of balanced graphs is hereditary, $H$ is balanced in this case.

\proofclaimspace The three cases above prove assertion~\eqref{item:1vertex}. Together with the proof of assertion~\eqref{item:2vertices}, this shows that, assuming $G$ is balanced and that both assertions hold for $T_1$ and $T_2$, both assertions hold for $T$. Since the base case was proved at the beginning, this completes the induction and the proof of the proposition.\end{proof}

The following theorem is the main result of this work. It characterizes balancedness within the class of distance-hereditary graphs both by the hereditary clique-Helly property and by a single forbidden induced subgraph.

\begin{theorem}\label{thm:DHbalancedGraphs} Let $G$ be a distance-hereditary graph. The following assertions are equivalent:
\begin{enumerate}[(i)]
\item\label{item:DH-balanced} $G$ is balanced.
\item\label{item:DH-HCH} $G$ is hereditary clique-Helly.
\item\label{item:DH-no-co3K2} $G$ contains no induced $\overline{3K_2}$.
\end{enumerate}
\end{theorem}
\begin{proof} \resetproofclaims The implication ${\eqref{item:DH-balanced}}\Rightarrow{\eqref{item:DH-HCH}}$ follows from Theorem~\ref{thm:balanced-HCH}. The implication ${\eqref{item:DH-HCH}}\Rightarrow{\eqref{item:DH-no-co3K2}}$ follows because $\overline{3K_2}$ is not clique-Helly. Indeed, if $aa'$, $bb'$, and $cc'$ are its nonedges, then the maximal cliques $\{a,b,c'\}$, $\{a,b',c\}$, and $\{a',b,c\}$ are pairwise intersecting and have empty total intersection.

It only remains to prove ${\eqref{item:DH-no-co3K2}}\Rightarrow{\eqref{item:DH-balanced}}$. By Theorem~\ref{thm:NewDefinitionDH}, it suffices to prove that, for each $\OVE$-tree $T$, if $G=\DHG(T)$ and $G$ contains no induced $\overline{3K_2}$, then $G$ is balanced. We proceed by induction on the definition of $\OVE$-trees.

If $T$ has exactly one vertex and $G=\DHG(T)$, then $G$ has one vertex and is balanced.

We now consider the case where $T=T_1\oplus_\alpha T_2$ for some $\alpha\in\{\flabel,\tlabel,\plabel\}$. Let $G=\DHG(T)$ and suppose that $G$ contains no induced $\overline{3K_2}$. Let $G_i=\DHG(T_i)$ and $X_i=\TS(T_i)$ for each $i\in\{1,2\}$. By Remark~\ref{rmk:TS-path}, $G_i$ is an induced subgraph of $G$ for each $i\in\{1,2\}$. In particular, neither $G_1$ nor $G_2$ contains an induced $\overline{3K_2}$. Hence, by the induction hypothesis applied to $T_i$, $G_i$ is balanced for each $i\in\{1,2\}$.

If $\alpha=\flabel$, then $G=G_1\cup G_2$. Since $G_1$ and $G_2$ are balanced and the maximal cliques of $G$ are precisely the maximal cliques of $G_1$ and $G_2$, $G$ is balanced. Thus, in the remainder of the proof, we only consider the case where $\alpha\in\{\tlabel,\plabel\}$. Let $v$ be a vertex not in $V(G)$. For each $i\in\{1,2\}$, let $H_i$ be the graph that arises from $G_i$ by adding $v$ such that $N_{H_i}(v)=X_i$. Thus, $V(H_1)\cap V(H_2)=\{v\}$ and, by Definition~\ref{def:OVE-graph}, $G=H_1\ast H_2$.

\proofclaim{$H_i$ is balanced for each $i\in\{1,2\}$}\label{claim:main-Hi-balanced} Let $i\in\{1,2\}$. By Remark~\ref{rmk:TS-path}, $X_{3-i}$ is nonempty. Let $x$ be a vertex of $X_{3-i}$. If $H_i$ contained an induced $\overline{3K_2}$ containing $v$, then replacing $v$ by $x$ would produce an induced $\overline{3K_2}$ in $G$, because $x$ has in $G$ the same neighbors in $V(G_i)$ as $v$ has in $H_i$. This contradicts the assumption that $G$ contains no induced $\overline{3K_2}$. Hence, $H_i$ contains no induced $\overline{3K_2}$ containing $v$. Since $G_i$ is balanced, Proposition~\ref{prop:TS-bal}\eqref{item:1vertex} applied to $T_i$ implies that $H_i$ is balanced. This proves the claim.

\proofclaimspace For each $i\in\{1,2\}$, let $H_i^+$ be the graph that arises from $H_i$ by adding a false twin $v'$ of $v$.

\proofclaim{For each $i\in\{1,2\}$, if $X_{3-i}$ is not a clique of $G_{3-i}$, then $H_i^+$ is balanced}\label{claim:main-Hi+-balanced} Let $i\in\{1,2\}$ and suppose that $X_{3-i}$ is not a clique of $G_{3-i}$. Let $x$ and $x'$ be two vertices of $X_{3-i}$ that are nonadjacent in $G_{3-i}$. If $H_i^+$ contained an induced $\overline{3K_2}$ containing $v$ and $v'$, then replacing $v$ and $v'$ by $x$ and $x'$, respectively, would produce an induced $\overline{3K_2}$ in $G$. This replacement is valid because $x$ and $x'$ are nonadjacent in $G$ and have in $G$ the same neighbors in $V(G_i)$ as $v$ and $v'$ have in $H_i^+$. This contradicts the assumption that $G$ contains no induced $\overline{3K_2}$. Hence, no induced $\overline{3K_2}$ of $H_i^+$ contains both $v$ and $v'$. By Proposition~\ref{prop:TS-bal}\eqref{item:2vertices} applied to $T_i$, $H_i^+$ is balanced. This proves the claim.

\proofclaimspace We now apply Proposition~\ref{prop:ast-bal} according to the three possible clique alternatives for $X_1$ in $G_1$ and $X_2$ in $G_2$.

\resetproofcases

\proofcase{$X_1$ and $X_2$ are both cliques of $G_1$ and $G_2$, respectively} In this case, since $H_1$ and $H_2$ are balanced by Claim~\ref{claim:main-Hi-balanced}, Proposition~\ref{prop:ast-bal}\eqref{item:TSareCliques} implies that $G$ is balanced.

\proofcase{For some $i\in\{1,2\}$, $X_i$ is a clique of $G_i$ and $X_{3-i}$ is not a clique of $G_{3-i}$} For this index $i$, Claim~\ref{claim:main-Hi+-balanced} implies that $H_i^+$ is balanced. Moreover, $H_{3-i}$ is balanced by Claim~\ref{claim:main-Hi-balanced}. Hence, Proposition~\ref{prop:ast-bal}\eqref{item:TS1isClique} applied to $H_i$ and $H_{3-i}$ in place of $H_1$ and $H_2$ implies that $H_i\ast H_{3-i}$ is balanced. Since $H_i\ast H_{3-i}=G$, $G$ is balanced.

\proofcase{Neither $X_1$ is a clique of $G_1$ nor $X_2$ is a clique of $G_2$} In this case, Claim~\ref{claim:main-Hi+-balanced} applied with $i=1$ and $i=2$ implies that $H_1^+$ and $H_2^+$ are balanced, respectively. Hence, by Proposition~\ref{prop:ast-bal}\eqref{item:TSareNotCliques}, $G$ is balanced.
\proofclaimspace Together with the case $\alpha=\flabel$ considered above, the three cases above complete the inductive step. Since the base case was proved above, this completes the induction and the proof of the theorem.\end{proof}

\section{Algorithmic aspects}\label{sec:algorithmics}

In this section, we give an explicit linear-time algorithm for deciding balancedness within the class of distance-hereditary graphs. Given a distance-hereditary graph, the algorithm either returns a vertex set inducing a $\overline{3K_2}$ as a certificate that the graph is not balanced or concludes that the graph is balanced.

By Theorem~\ref{thm:DHbalancedGraphs}, balancedness of a distance-hereditary graph can be recognized in $O(m^2+n)$ time by applying the algorithm of Lin and Szwarcfiter~\cite{ReconocimientoHCH} for recognizing hereditary clique-Helly graphs, where $n$ denotes the number of vertices and $m$ the number of edges of the graph. Alternatively, since every distance-hereditary graph has clique-width at most $3$ and a corresponding $3$-expression can be built in linear time~\cite{MR1792124}, the theorem of Courcelle et al.~\cite{MR1739644} on $\mathrm{LinEMSOL}(\tau_1)$-definable optimization problems yields a linear-time algorithm that, given a distance-hereditary graph, finds a vertex set inducing $\overline{3K_2}$ or concludes that none exists. However, the constant factors hidden in the running time of such an algorithm are huge, even for small clique-width bounds, and this is unavoidable for general monadic second-order formulas~\cite{MR2536468,MR2092847}. Our algorithm works directly on an $\OVE$-tree and keeps the certificate construction explicit.

We begin by defining a tuple of auxiliary certificates associated with an $\OVE$-tree. Let $T$ be an $\OVE$-tree, let $G=\DHG(T)$, and let $X=\TS(T)$. A \emph{certificate tuple for $T$} is any quadruple $(N,C,W,S)$ whose components satisfy the following assertions:
\begin{enumerate}[(i)]
\item\label{it:certN} $N=\nil$ if and only if $X$ is a clique of $G$; otherwise, $N$ is a pair of nonadjacent vertices in $G[X]$.
\item\label{it:certC} $C=\nil$ if and only if $G[X]$ contains no induced $C_4$; otherwise, $C$ is the vertex set of an induced $C_4$ in $G[X]$.
\item\label{it:certW} $W=\nil$ if and only if $G$ contains no induced $W_4$ in which the vertices of the $4$-cycle lie in $X$ and the universal vertex lies in $V(G)-X$; otherwise, $W$ is the vertex set of such an induced $W_4$ in $G$.
\item\label{it:certS} $S=\nil$ if and only if $G$ contains no induced $\overline{3K_2}$; otherwise, $S$ is the vertex set of an induced $\overline{3K_2}$ in $G$.
\end{enumerate}

We now show that a certificate tuple for $T$ can be computed in linear time by Algorithm~\ref{algo:certificates}. We denote its output on input $T$ by $\cert(T)$.

\begin{algorithm}
\DontPrintSemicolon
\KwIn{An $\OVE$-tree $T$.}
\KwOut{A certificate tuple for $T$.}

\lIf{$\vert V(T)\vert=1$}{\KwRet{$(\nil,\nil,\nil,\nil)$}}\nllabel{line:base}
\Else{ 
Let $T_1$, $T_2$, and $\alpha$ be such that $T=T_1\oplus_\alpha T_2$ and let $v_i$ be the root of $T_i$ for each $i\in\{1,2\}$\nllabel{line:split}\;
$( N^{(1)}, C^{(1)}, W^{(1)}, S^{(1)})\gets \cert(T_1)$\nllabel{line:rec1}\;
$( N^{(2)}, C^{(2)}, W^{(2)}, S^{(2)})\gets \cert(T_2)$\nllabel{line:rec2}\;

\If{$\alpha=\flabel$}{

$N \gets\{v_1, v_2\}$\nllabel{line:f1}\;
\leIf{$C^{(1)}\neq \nil$}{$C \gets C^{(1)}$}{$C\gets C^{(2)}$}\nllabel{line:f2}
\leIf{$W^{(1)}\neq \nil$ }{$W\gets W^{(1)}$}{$W\gets W^{(2)}$}\nllabel{line:f3} 
\leIf{$S^{(1)}\neq \nil$}{$S\gets S^{(1)}$}{$S\gets S^{(2)}$}\nllabel{line:f4}

}\Else{
\If{$\alpha=\tlabel$}{

\leIf{$N^{(1)}\neq \nil$}{$N\gets N^{(1)}$}{$N\gets N^{(2)}$}\nllabel{line:t1}

\lIf{$C^{(1)}\neq \nil$ }{$C\gets C^{(1)}$}\nllabel{line:t2}
\lElseIf{$C^{(2)}\neq \nil$}{ 
$C\gets C^{(2)}$}\nllabel{line:t3}
\lElseIf{$N^{(1)}\neq \nil$ \KwAnd $N^{(2)}\neq \nil$}{
$C\gets {N^{(1)}}\cup{N^{(2)}}$
}\nllabel{line:t4}
\lElse{$C \gets \nil$}\nllabel{line:t5}

\leIf{$W^{(1)}\neq \nil$ }{$W\gets W^{(1)}$}{ $W\gets W^{(2)}$}\nllabel{line:t6}

}\ElseIf{$\alpha=\plabel$}{
$N \gets N^{(1)}$\nllabel{line:p1}\;
$C \gets C^{(1)}$\nllabel{line:p2}\;
\leIf{$C^{(1)}\neq \nil$}{$W\gets {C^{(1)}}\cup{\{v_2\}}$}{$W\gets\nil$}\nllabel{line:p3}

}

\lIf{$S^{(1)}\neq \nil$}{$S \gets S^{(1)}$}\nllabel{line:s1}
\lElseIf{$S^{(2)}\neq \nil$}{$S \gets S^{(2)}$}\nllabel{line:s2}
\lElseIf{$W^{(1)}\neq \nil$}{$S\gets {W^{(1)}}\cup{\{v_2\}}$}\nllabel{line:s3}
\lElseIf{$W^{(2)}\neq \nil$}{$S\gets {W^{(2)}}\cup{\{v_1\}}$}\nllabel{line:s4}
\lElseIf{$C^{(1)}\neq \nil$ \KwAnd $N^{(2)}\neq \nil$}{$S\gets {C^{(1)}}\cup{N^{(2)}}$}\nllabel{line:s5}
\lElseIf{$C^{(2)}\neq \nil$ \KwAnd $N^{(1)}\neq \nil$}{$S\gets {C^{(2)}}\cup{N^{(1)}}$}\nllabel{line:s6}
\lElse{$S\gets\nil$}\nllabel{line:s7}

}

\KwRet{$(N, C, W, S)$}\nllabel{line:return}
}

\caption{$\cert(T)$}\label{algo:certificates}
\end{algorithm}

\begin{lemma}\label{lemma:AlgoCertificates} For every $\OVE$-tree $T$, the output $\cert(T)$ of Algorithm~\ref{algo:certificates} is a certificate tuple for $T$. Moreover, Algorithm~\ref{algo:certificates} can be performed in linear time.\end{lemma}
\begin{proof} We first prove that $\cert(T)$ is a certificate tuple for $T$. Let $T$ be an $\OVE$-tree. We proceed by induction on the definition of $\OVE$-trees. Let $G=\DHG(T)$ and $X=\TS(T)$.

If $T$ has exactly one vertex, then $G$ has one vertex. Hence, the algorithm returns $(\nil,\nil,\nil,\nil)$, which is a certificate tuple for $T$.

Suppose now that $T=T_1\oplus_\alpha T_2$ for some $\alpha\in\{\flabel,\tlabel,\plabel\}$. Let $G_i=\DHG(T_i)$ and $X_i=\TS(T_i)$ for each $i\in\{1,2\}$. By the induction hypothesis, the recursive calls in lines~\ref{line:rec1} and~\ref{line:rec2} of Algorithm~\ref{algo:certificates} return certificate tuples $\cert(T_1)$ and $\cert(T_2)$, respectively. We will prove that the output $\cert(T)$ of Algorithm~\ref{algo:certificates} is a certificate tuple for $T$.
\resetproofclaims
\resetproofcases

\proofclaim{Suppose that $\alpha\in\{\tlabel,\plabel\}$. If $S$ is the vertex set of an induced $\overline{3K_2}$ in $G$ and $S$ is contained in neither $V(G_1)$ nor $V(G_2)$, then, for some $i\in\{1,2\}$, either $S=C\cup N$, where $C$ is the vertex set of an induced $C_4$ in $G_i[X_i]$ and $N$ is a pair of nonadjacent vertices of $G_{3-i}[X_{3-i}]$, or $S=Y\cup\{x\}$, where $Y$ induces in $G_i$ a $W_4$ in which the vertices of the $4$-cycle lie in $X_i$, the universal vertex lies in $V(G_i)-X_i$, and $x\in X_{3-i}$}\label{claim:cert-cross-obst} Let $S_i=S\cap V(G_i)$ for each $i\in\{1,2\}$. Suppose first that $\vert S_1\vert\geq 2$ and $\vert S_2\vert\geq 2$. Thus, for each $i\in\{1,2\}$, each vertex of $S_i$ has at least one neighbor in $S_{3-i}$ and, necessarily, $S_i\subseteq X_i$. Hence, $S_1$ is complete to $S_2$ in $G$. Consequently, for some $i\in\{1,2\}$, $S_i$ induces a $C_4$ in $G_i[X_i]$ and $S_{3-i}$ is a pair of nonadjacent vertices of $G_{3-i}[X_{3-i}]$. Taking $C=S_i$ and $N=S_{3-i}$, the first alternative holds. By symmetry, it only remains to consider the case where $\vert S_2\vert=1$. Let $x$ be the unique vertex of $S_2$. Since $x$ has exactly one nonneighbor in $S_1$ and is adjacent to the other four vertices of $S_1$, $x$ belongs to $X_2$, its nonneighbor in $S_1$ lies in $V(G_1)-X_1$, and its four neighbors in $S_1$ lie in $X_1$. These four vertices induce a $C_4$, and the remaining vertex of $S_1$ is adjacent to all of them. Thus, taking $Y=S_1$, the second alternative holds.

We consider three cases according to the value of $\alpha$.

\proofcase{$\alpha=\flabel$} Since $v_i$ is the root of $T_i$ for each $i\in\{1,2\}$, Remark~\ref{rmk:TS-path} implies that $v_i\in X_i$ for each $i\in\{1,2\}$. Moreover, since $G$ is the disjoint union of $G_1$ and $G_2$ and $X=X_1\cup X_2$, $v_1$ and $v_2$ are nonadjacent vertices in $G[X]$ and every induced $C_4$ in $G[X]$, every induced $W_4$ of the type described in \eqref{it:certW}, and every induced $\overline{3K_2}$ in $G$ is contained in either $G_1$ or $G_2$. Hence, the assignments in lines~\ref{line:f1}--\ref{line:f4} make assertions~\eqref{it:certN}--\eqref{it:certS} in the definition of a certificate tuple for $T$ hold.

\proofcase{$\alpha=\tlabel$} Since every nonedge of $G[X]$ is contained in $G_1[X_1]$ or in $G_2[X_2]$, line~\ref{line:t1} makes \eqref{it:certN} hold. Also, any vertex set inducing a $C_4$ in $G[X]$ is either contained in $X_1$, contained in $X_2$, or is the union of a pair of nonadjacent vertices of $G_1[X_1]$ and a pair of nonadjacent vertices of $G_2[X_2]$. Hence, lines~\ref{line:t2}--\ref{line:t5} make \eqref{it:certC} in the definition of a certificate tuple for $T$ hold. We now prove that line~\ref{line:t6} makes \eqref{it:certW} in the definition of a certificate tuple for $T$ hold. Suppose first that $W^{(i)}\neq\nil$ for some $i\in\{1,2\}$. By the induction hypothesis, $W^{(i)}$ is the vertex set of an induced $W_4$ in $G_i$ in which the vertices of the $4$-cycle lie in $X_i$ and the universal vertex lies in $V(G_i)-X_i$. Since $G_i$ is an induced subgraph of $G$, $X_i\subseteq X$, and $V(G_i)-X_i\subseteq V(G)-X$, if we let $W=W^{(i)}$, then $W$ satisfies \eqref{it:certW} in the definition of a certificate tuple for $T$. Conversely, let $W\neq\nil$ be as in \eqref{it:certW} in the definition of a certificate tuple for $T$. If the vertices of the $4$-cycle in $G[W]$ were not contained in $X_i$ for some $i\in\{1,2\}$, then, since $X_1$ is complete to $X_2$, two of them would lie in $X_1$ and the other two in $X_2$. But the universal vertex of $G[W]$ lies in $V(G)-X$ and has neighbors in at most one of $X_1$ and $X_2$, a contradiction. Thus, the vertices of the $4$-cycle in $G[W]$ are contained in $X_i$ for some $i\in\{1,2\}$, and thus the universal vertex of $G[W]$ must lie in $V(G_i)-X_i$. Consequently, $W^{(i)}\neq\nil$. Hence, such a vertex set $W$ exists if and only if $W^{(1)}\neq\nil$ or $W^{(2)}\neq\nil$, and line~\ref{line:t6} makes \eqref{it:certW} in the definition of a certificate tuple for $T$ hold. Finally, Claim~\ref{claim:cert-cross-obst} shows that lines~\ref{line:s1}--\ref{line:s7} make \eqref{it:certS} in the definition of a certificate tuple for $T$ hold.

\proofcase{$\alpha=\plabel$} Since $X=X_1$, a nonedge of $G[X]$ is exactly a nonedge of $G_1[X_1]$, and an induced $C_4$ in $G[X]$ is exactly an induced $C_4$ in $G_1[X_1]$. Hence, the assignments in lines~\ref{line:p1} and~\ref{line:p2} make \eqref{it:certN} and~\eqref{it:certC} in the definition of a certificate tuple for $T$ hold. We now prove that line~\ref{line:p3} makes \eqref{it:certW} in the definition of a certificate tuple for $T$ hold. Since $X=X_1$, the vertices of the $4$-cycle of any induced $W_4$ of the type described in \eqref{it:certW} in the definition of a certificate tuple for $T$ lie in $X_1$ and therefore induce a $C_4$ in $G_1[X_1]$. Hence, by the induction hypothesis, $C^{(1)}\neq\nil$. Conversely, if $C^{(1)}\neq\nil$, then $C^{(1)}$ is the vertex set of an induced $C_4$ in $G_1[X_1]$. Since $X_1$ is complete to $X_2$ in $G$ and $v_2\in V(G)-X$, the set $C^{(1)}\cup\{v_2\}$ induces the required $W_4$ in $G$. Hence, line~\ref{line:p3} makes \eqref{it:certW} in the definition of a certificate tuple for $T$ hold. Finally, Claim~\ref{claim:cert-cross-obst} shows that lines~\ref{line:s1}--\ref{line:s7} make \eqref{it:certS} in the definition of a certificate tuple for $T$ hold.
\proofclaimspace It only remains to prove the linear-time bound. Each tree considered by the recursion is represented by its root and by the first child of the root that belongs to that tree, if any. Line~\ref{line:base} takes constant time. In addition, line~\ref{line:split} can be performed in constant time because $v_1$ is the root of $T$, $v_2$ is the first child of $v_1$ in $T$, $\alpha$ is the label of the edge $v_1v_2$, the tree $T_2$ is determined by the root $v_2$ and by the first child of $v_2$, if any, while $T_1$ is determined by the root $v_1$ and by the next sibling of $v_2$ among the children of $v_1$, if any. Moreover, lines~\ref{line:f1}--\ref{line:return} can be performed in constant time because all involved vertex sets have size at most six. By the induction hypothesis, the recursive calls in lines~\ref{line:rec1} and~\ref{line:rec2} take linear time in $T_1$ and $T_2$, respectively. Thus, the total time is linear in $T$. This completes the proof of the lemma.\end{proof}

The structural characterization in Theorem~\ref{thm:DHbalancedGraphs} turns Algorithm~\ref{algo:certificates} into a recognition algorithm for balancedness of distance-hereditary graphs. Given a distance-hereditary graph $G$, build an $\OVE$-tree $T$ of $G$, compute $\cert(T)$, let $S$ be its last component, and either return that $G$ is not balanced together with $S$ if $S\neq\nil$, or return that $G$ is balanced otherwise.

\begin{theorem}\label{thm:DH-balanced-algo} There is a linear-time algorithm that, given a distance-hereditary graph $G$, decides whether $G$ is balanced. Moreover, if $G$ is not balanced, the algorithm returns a vertex set inducing a $\overline{3K_2}$ in $G$.\end{theorem}
\begin{proof} Let $G$ be a distance-hereditary graph and consider the algorithm described above. By Theorem~\ref{thm:AlgoRecognitionDH}, an $\OVE$-tree $T$ of $G$ can be built in linear time. By Lemma~\ref{lemma:AlgoCertificates}, $\cert(T)$ can also be computed in linear time and its last component $S$ satisfies \eqref{it:certS}. Since $G$ is distance-hereditary, Theorem~\ref{thm:DHbalancedGraphs} implies that $G$ is balanced if and only if $S=\nil$. Hence, the whole algorithm is correct and can be completed in linear time.\end{proof}

\section*{Acknowledgements}

L.\ Busolini and G.\ Dur\'an were partially supported by UBACyT Grant 20020170100495BA, PIP-CONICET Grant 11220200100084CO, and ANPCyT-PICT Grant 2021-I-A-00755. G.\ Dur\'an was partially supported by ISCI, Chile (ICM-FIC: P05-004-F, CONICYT: FB0816). M.D.\ Safe was partially supported by CONICET Grant PIBAA 28720210101185CO and UNS Grant PGI 24/L133.

\bibliographystyle{abbrv}
\bibliography{biblioDHbal}

\end{document}